\documentclass[11pt]{amsart}

\title{Ricci flow on Riemannian groupoids}
\author{Christian Hilaire}

\usepackage{amsmath,amsthm,amssymb}

\newtheorem{thm}{Theorem}

\newtheorem{ppr}[thm]{Proposition}

\newtheorem{lma}[thm]{Lemma}

\newtheorem{crly}[thm]{Corollary}

\theoremstyle{remark}
\newtheorem{rema}[thm]{Remark}

\theoremstyle{remark}
\newtheorem{exam}[thm]{Example}

\begin{document}

\begin{abstract}
We study the Ricci flow on Riemannian groupoids. We assume that these groupoids are closed in the sense of \cite{lott}[Section 5] and that the space of orbits is compact and connected. We prove the short time existence and uniqueness of the Ricci flow on these groupoids. We also define a $\mathcal{F}$-functional and derive the corresponding results for steady breathers on these spaces.
\end{abstract}
 
\maketitle

\section{Introduction}

	     One of the basic methods used in Ricci flow and in other geometric evolution equations is the blowup or blowdown analysis. This method consists of constructing a sequence of pointed complete solutions of the Ricci flow by a parabolic rescaling of the original solution. If this sequence has uniform bounds on the curvature and if there is a uniform lower bound on the injectivity radius of the basepoints at time zero, then this sequence will converge in a smooth sense to a pointed complete Ricci flow solution \cite{hamilton1}. The condition on the injectivity radius insures that the sequence does not collapse to a lower dimensional manifold. One of Perelman's key results was showing the existence of such a lower bound in the case of finite time singularities \cite{perelman}. On the other hand, if we look at Ricci flow solutions that exists for all $t \geq 0$ and try to determine the behavior as $t \rightarrow \infty $ by performing a blowdown analysis, we don't necessarily have a lower bound on the injectivity radius. Thus, the constructed sequence can collapse. However, as shown in \cite{lott}, the limit can be described as a Ricci flow on a pointed closed Riemannian groupoid of the same dimension. A Riemannian groupoid is an \'etale groupoid $(\mathcal{G},M)$ with a Riemannian metric $g$ on the space of units $M$ such that the local diffeomorphisms generated by the groupoid elements are local isometries. In other words, the metric is invariant under the action of $\mathcal{G}$ on $M$.
			
			 The study of the Ricci flow on groupoids can help us understand the long-time behavior of the Ricci flow on smooth manifolds. Such an approach was used in \cite{hilaire} to study the long time behavior of the Ricci tensor for an immortal solution with unifomly bounded curvature and diameter. In this paper, we provide some fundamental results. We consider the case where the space of orbits of the groupoid is compact and connected. Our first result is the short time existence and uniqueness of a solution of the Ricci flow.

\begin{thm}

Suppose $(\mathcal{G},M)$ is a closed groupoid with compact connected space of orbits $W=M/\mathcal{G}$ and suppose $g_{0}$ is a $\mathcal{G}$-invariant Riemannian metric on $M$. Then there exists a unique solution $g(t)$ of the Ricci flow defined on a maximal interval $[0,T)$ and with initial condition $g(t)=g(0)$. Furthermore, if $T<\infty$, then 
\begin{eqnarray*}
	sup_{M} |Rm|(.,t) \rightarrow \infty
	\end{eqnarray*}
	as $t$ approaches $T$.
	\end{thm}
	As in the manifold case, we will prove the short time existence of the Ricci flow by proving the short time existence of the Ricci-DeTurck flow. The proof of the short-time existence is not a straightforward adaptation of the usual proof because the unit space
of the groupoid need not be a complete Riemannian manifold. Instead, we consider the frame bundle associated to a background Riemannian metric, quotient this space by the lifted groupoid action and construct an auxiliary parabolic partial differential equation.
     
		To prove the uniqueness of the Ricci flow, we will avoid dealing with the theory of harmonic map heat flow and instead use the energy functional approach introduced by Kotschwar in \cite{kotschwar}. We prove the result in greater generality, namely
for Riemannian groupoids of bounded curvature with a complete orbit space and that satisfy a mild growth condition on the volume. The uniqueness result was used in \cite{lott3} in order to prove results about spatial asymptotics for Ricci flow on some noncompact manifolds.

	The space of ``arrows'' $\mathcal{G}$ of the closed Riemannian groupoid is equipped with two topologies. One that gives it the structure of an \'etale groupoid and another one that gives it the structure of a proper Lie groupoid. We construct a family $\Xi$ of smooth Haar systems for this proper Lie groupoid. Recall that a Haar system is a collection of measures $\{ d\rho^{p} \}_{p \in M}$ where, for each $p \in M$, the measure has support on $r^{-1}(p)=\mathcal{G}^{p}$, the preimage of $p$ under the range map of the groupoid. These measures are $\mathcal{G}$-invariant in the appropriate sense and satisfy an additional property. For simplicity, we shall sometimes denote a Haar system by $\rho$ or by $d\rho$.
	
	The family $\Xi$ of Haar systems that we construct can be identified with $C^{\infty}_{\mathcal{G}}(M)$, the space of smooth $\mathcal{G}$-invariant functions on the space of units $M$, as follows. For a fixed $\rho_{0}$ of $\Xi$, we have the bijection
\begin{eqnarray*}
	C^{\infty}_{\mathcal{G}}(M) & \rightarrow & \Xi \\
	 w                & \rightarrow  & e^{w} d\rho_{0}.
\end{eqnarray*}

	This identification is clearly not natural since it depends on the choice of the representative $\rho$. However, we use it to prove certain results in this paper.
	
	Furthermore, to each element of $\Xi$, we will associate a $\mathcal{G}$-invariant closed one-form $\theta$ that comes from a Riemannian submersion generated by the closed groupoid.  We prove the following 
\begin{thm} \label{monotonos}
Suppose $(\mathcal{G},M)$ is a closed groupoid with compact connected space of orbits $W=M/\mathcal{G}$. Let $\mathcal{M}_{\mathcal{G}}$ be the space of $\mathcal{G}$-invariant Riemannian metrics on $M$ and let $\Xi$ be the family of smooth Haar systems mentioned above. We define the $\mathcal{F}$-functional $\mathcal{F}:  \mathcal{M}_{\mathcal{G}} \times \Xi  \rightarrow \mathbb{R}$ by
\begin{eqnarray}
\mathcal{F}(g, \rho) =\int_{M} \left(R + |\theta |^2 \right)\,  \phi^2d\mu_{g}  \label{Ffunctionalprime}
\end{eqnarray}
where $R$ is the scalar curvature of the metric $g$, $d\mu_{g}$ the Riemannian density, $\theta$ the one-form associated to the Haar system and $\phi$ any smooth nonnegative cutoff function for the Haar system.                      
	Under the flow equations 
\begin{eqnarray}
\frac{\partial g}{\partial t}&=& -2Ric(g)  \label{monotonoseq} \\
\frac{\partial (d\rho)}{\partial t} &=& \left(|\theta|^2 -R-\mathrm{div}(\theta) \right) d\rho \notag
\end{eqnarray}
the evolution of the $\mathcal{F}$-functional is given by
\begin{eqnarray}
\frac{d}{dt}\mathcal{F}(g(t),\rho(t))=2 \int_{M} |Ric(g) + \frac{1}{2}\mathcal{L}_{\theta^{\sharp}}g|^2 \, \phi^2 d\mu_{g} \geq 0.
\end{eqnarray}
\end{thm}
As is shown in \cite{gorokhovsky}, the integrals listed above do not depend on the choice of the nonnegative cutoff function for the Haar system.

The paper is structured as follows. In section 2, we collect and prove some results about closed Riemannian groupoids. In particular, we go over the construction of the cross-product groupoid on the orthonormal frame bundle as outlined in \cite{gorokhovsky}[Section 2]. This construction plays a major role in most of the proofs in this paper. We also  define the family $\Xi$ of Haar systems that was mentioned above. In Section 3, we prove the uniqueness of the Ricci flow solution and, in Section 4, we prove the short time existence. Section 5 is devoted to the $\mathcal{F}$-functional and the corresponding results for steady breathers. Finally, we include an appendix where we briefly study the case of $2$-dimensional closed Riemannian groupoids.

 After the first version of this paper was posted on the arXiv, Bedulli-He-Vezzoni posted a preprint with independent but related results \cite{bedulli}.
	
 \textbf{Acknowledgements:} I would like to thank my adviser John Lott for helpful discussions and comments on earlier drafts of this paper.

\section{Riemannian groupoids}  \label{sec:groupoids}
In this section, we list the necessary material that we will need about groupoids. For information about the basic definitions and constructions in the theory of groupoids, we refer to the books \cite{bridson,moerdijk}. 

           In Subsection 2.1, we introduce some notation and review the concept of a complete effective Riemannian (\'etale) groupoid. In Subsection 2.2, we construct the ``closure'' of such a groupoid as defined in \cite{haefliger,salem1}. This will give us a groupoid equipped with two topologies, one which gives it the structure of a proper Lie groupoid and a finer one which gives it the structure of a complete effective Riemannian groupoid. In Subsection 2.3, we briefly recall Haefliger's local models for the structure of a closed Riemannian groupoid.
					
					In the next three subsections, we go over the same constructions done in \cite{gorokhovsky}[Sections 2 and 3]. We construct the cross-product groupoid on the orthonormal frame bundle and use it to generate a Haar system for the closed groupoid and the associated mean curvature form. We equip this Haar system with a nonnegative cutoff function. These quantities are used in Subsection 2.6 to define a measure for integrating functions on the space of units which are invariant under the groupoid action. We also derive an integration by parts formula for this measure. Then, in Subsection 2.7, we define the family of Haar sytems $\Xi$ and derive some relevant evolution equations for the Haar system and the mean curvature form. These evolution equations will be used when we will differentiate the energy functional in Section 3 and the $\mathcal{F}$-functional in Section 6.
					
					The last subsection is devoted to $\mathcal{G}$-paths and $\mathcal{G}$-invariant vector fields. We recall how the smooth $\mathcal{G}$-paths can be used to define a length structure on the space of orbits and show that any $\mathcal{G}$-invariant vector field defines a one-parameter family of differentiable equivalences of the Riemannian groupoid in the sense of \cite{haefliger}[Subsection 1.4].

\subsection{Riemannian groupoids} \label{subsec:groupoid-basics}
We will denote a smooth \'etale groupoid by a pair $(\mathcal{G},M^{n})$ where $M^{n}$ is the space of units and $\mathcal{G}$ is the space of ``arrows''. The corresponding source and range maps will be denoted by $s$ and $r$ respectively. Our convention is that, for any $h_1$ and $h_2 \in \mathcal{G}$, the product $h_1h_2$ is defined if and only if $s(h_1)=r(h_2)$. Given any $x \in M$, we write $\mathcal{G}^{x}$ for $r^{-1}(x)$, $\mathcal{G}_{x}$ for $s^{-1}(x)$, $\mathcal{G}^{x}_{x}$ for the isotropy group $s^{-1}(x) \cup r^{-1}(x)$ of $x$ and $O_{x}$ for the orbit of $x$. Furthermore, given any $h \in \mathcal{G}$, we will denote by $dh: T_{s(h)}M \rightarrow T_{r(h)}M$ the linearization of $h$.  Recall that this is simply the differential at $s(h)$ of a diffeomorphism $\phi:U \rightarrow V$ where $U$ and $V$ are open neighborhoods of $s(h)$ and $r(h)$ respectively and the map $\phi$ has the form $r\circ \alpha$ for some local section $\alpha:U \rightarrow \mathcal{G}$ of $s$ at $h$. More precisely, the map $\alpha$ satisfies $s \circ \alpha=id $ and $\alpha(s(h))=h$. We will consider \'etale groupoids up to \emph{weak equivalence} (\cite{moerdijk}[Chapter 5.4]). A smooth \'etale groupoid is said to be \emph{developable} if it is weakly equivalent to the \'etale groupoid associated to the action of a discrete group on a smooth manifold through diffeomorphisms.

			Given a smooth manifold $M$, let $\Gamma(M)$ be the space of germs of diffeomorphisms $\phi:U \rightarrow V$  between open subsets of $M$. Then, for any smooth \'etale groupoid $(\mathcal{G},M)$, we have the map
	\begin{eqnarray*}
	\kappa:\mathcal{G} &\rightarrow & \Gamma(M) \\
	    h       & \rightarrow & germ_{s(h)}( r\circ \alpha)
  \end{eqnarray*}
	where, as before, the map $\alpha$ is a local section of $s$ at $h$. The \'etale groupoid  $(\mathcal{G},M)$ is said to be \emph{effective} if the map $\kappa$ is injective.
	
	We will also say that the \'etale groupoid  $(\mathcal{G},M)$ is \emph{complete} if, for every pair of points $x$ and $y$ of $M$, there are neighborhoods $U$ and $V$ of $x$ and $y$ respectively so that, for any $h \in \mathcal{G}$ with $s(h) \in U$ and $r(h) \in V$, there is a corresponding section of $s$ defined on all of $U$.
	   	    
  Finally, a smooth \'etale groupoid $(\mathcal{G},M)$ is \emph{Riemannian} if $M$ is equipped with a Riemannian metric $g$ such that the elements of the pseudogroup generated by $\mathcal{G}$ are local isometries. Equivalently, for each $h \in \mathcal{G}$, the linearization $dh: (T_{s(h)}M,g_{s(h)}) \rightarrow (T_{r(h)}M,g_{r(h)})$ is an isometry of inner product spaces. We will then say that the metric $g$ is $\mathcal{G}$-invariant and we will denote the Riemannian groupoid by the triple $(\mathcal{G}, M,g)$. From now on, we will assume that the Riemannian groupoids are complete and effective unless otherwise stated.

\subsection{The closure of a Riemannian groupoid} \label{subsec:closure-of-groupoids}
 We will now recall how to construct the closure of a Riemannian groupoid in the sense of \cite{haefliger,salem1}. Given a smooth manifold $M^{n}$, we let $J^1(M)$ be the space of 1-jets of local diffeomorphisms equipped with the 1-jet topology. Elements of $J^1(M)$ can be represented by triples $(x,y,A)$ where $x, y \in M$ and $A: T_{x}M  \rightarrow T_{y}M$ is a linear bijection. The pair $(J^1(M),M)$ is a smooth Lie groupoid in the sense of \cite{moerdijk}[Chapter 5] with source and range maps defined by $s(x,y,A)=x$ and $r(x,y,A)=y$ respectively. It is not \'etale unless $\dim(M)=0$. A Riemannian metric $g$ on $M$ induces a smooth Lie subgroupoid $(J^1_{g}(M),M)$ of $(J^1(M),M)$  where $J^1_{g}(M) \subset J^1(M)$ consists of triples $(x,y,A)$ where the map $A$ is now a linear isometry with respect to $g$.

     A Lie groupoid $(\mathcal{G},M)$ is said to be \emph{proper} if the space $\mathcal{G}$ is Hausdorff and the map $(s,r): \mathcal{G} \rightarrow M \times M$ is proper. In \cite{gorokhovsky}[Lemma 1], it was shown that the Lie groupoid $(J^1_{g}(M),M)$ generated by a Riemannian metric on $M$ is proper. Indeed, the space $J^1_{g}(M)$ is Hausdorff and the map $(s,r): J^1_{g}(M) \rightarrow M\times M$ is proper since it defines a fiber bundle structure with fiber diffeomorphic to the compact Lie group $O(n)$.
		
  Given a Riemannian groupoid $(\mathcal{G},M,g)$, there is a natural map of $\mathcal{G}$ into $J^1(M)$ defined by
\begin{align*}
\nu:\mathcal{G} \rightarrow & J^1(M) \\
    h             \rightarrow & \left( s(h),r(h), dh \right).
\end{align*}
Since the metric $g$ is $\mathcal{G}$-invariant, the image of this map is contained in $J^1_{g}(M)$. Furthermore, the map $\nu$ is injective. This follows from our assumption that the Riemannian groupoid is effective and the fact that the germ of a local isometry is completely determined by its 1-jet. The space $\mathcal{G}$ can now be viewed as a subset of $J^1(M)$ and we can consider its closure $\overline{\mathcal{G}} \subset J^1(M)$ with respect to the 1-jet topology. The pair $( \overline{\mathcal{G}},M)$ is a topological subgroupoid of $(J^1_{g}(M),M)$. Since $\overline{\mathcal{G}}$ is contained in $J^1_{g}(M)$, it follows from \cite{salem2}[Section 2] that $( \overline{\mathcal{G}},M)$ is in fact a smooth Lie subgroupoid of $(J^1_{g}(M),M)$. The orbits of this Lie groupoid are closed submanifolds of $M$ and correspond to the closures of the orbits of the groupoid $(\mathcal{G},M)$ (see \cite{salem1}[Section 3]). It also follows that the quotient space $W= M/ \overline{\mathcal{G}}$, equipped with the quotient topology, is a Hausdorff space. The corresponding quotient map will be denoted by $\sigma: M \rightarrow W$. Finally, the fact that $\overline{\mathcal{G}}$ is a subset of $J^1_{g}(M)$ implies that the Lie groupoid $( \overline{\mathcal{G}},M)$ is proper.

     In addition to the subspace topology induced by the 1-jet topology on $J^1(M)$, the space $\overline{\mathcal{G}}$ can be equipped with a finer topology which gives the triple $(\overline{\mathcal{G}},M,g)$ the structure of a complete effective Riemannian groupoid. For now on, we will assume that the Riemannian groupoid $(\mathcal{G},M,g)$ satisfy $\overline{\mathcal{G}}=\mathcal{G}$. We will then say that the groupoid is \emph{closed}. Also, when we will work with the space $\mathcal{G}$, we will assume, unless otherwise stated, that it is equipped  with the subspace topology induced by the 1-jet topology. We will call the quotient space $W$ the \emph{space of orbits} and we will in general assume that it is compact and connected.
		
\subsection{Local structure of the closed groupoid} \label{localModel}
 Let $(\mathcal{G}, M,g)$ be a closed Riemannian groupoid as before and let $\mathfrak{g}_{M}$ be the corresponding Lie algebroid as defined in \cite{moerdijk}[Chapter 6]. As mentioned in \cite{gorokhovsky}[Subsection 2.5], $\mathfrak{g}_{M}$ is a $\mathcal{G}$-equivariant flat vector bundle over $M$  whose fibers are copies of a fixed Lie algebra $\mathfrak{g}$ called the \textbf{structural Lie algebra}. If $an: \mathfrak{g}_{M} \rightarrow TM$ is the anchor map of the Lie algebroid and if $P: U \times \mathfrak{g} \rightarrow \mathfrak{g}_{M}$ is a local parallelization over an open set $ U \subset M$, then the map $an \circ P$ describes a Lie algebra of Killing vector fields on $U$ ismorphic to $\mathfrak{g}$.

   For a fixed point $p \in M$, let $K$ denote the isotropy group of the closed groupoid at $p$. $K$ is a compact Lie group and we will denote its Lie algebra by $\mathfrak{k}$. Since $\mathfrak{g}$ can be realized as the Lie algebra of Killing vector fields on some neighborhood of $p$, there is a natural inclusion $i: \mathfrak{k} \rightarrow \mathfrak{g}$.
Furthermore, there is a representation $Ad: K \rightarrow Aut(\mathfrak{g})$ such that 
\begin{enumerate}
\item $ Ad $ leaves invariant the subalgebra $\mathfrak{k}=i(\mathfrak{k})$ of $\mathfrak{g}$ and $Ad|_{\mathfrak{k}}$ is the adjoint representation of $K$ on $\mathfrak{k}$.
 \item The differential of $Ad$ is the adjoint representation of $\mathfrak{k}$ on $\mathfrak{g}$.
 \end{enumerate}
Finally, there is a representation $T: K \rightarrow Gl(V)$ given by the action of $K$ on the subspace $V$ of $T_pM$ which is the orthogonal complement of the tangent space $T_{p}O_p$ of the orbit space $O_p$. The tangent space $T_{p}O_p$ is isomorphic to the vector space $\mathfrak{g}/\mathfrak{k}$ and the linear representation of $K$ on $T_{p}M$, which is faithful, corresponds to the direct sum representation $T \oplus Ad$ on $V \oplus \mathfrak{g}/\mathfrak{k}$.

It is shown in \cite{haefliger}[Section 4] that the quintuple $(\mathfrak{g},K,i,Ad,T)$ determines the weak equivalence class of the restriction of $\mathcal{G}$ (with the \'etale topology) to a a small invariant neighborhood of the orbit $O_{p}$. Furthermore, given such a quintuple, we can construct an explicit local model for the structure of the groupoid.

\subsection{The induced groupoid over the orthonormal frame bundle.}  \label{subsec:induced-groupoid}
Given a Riemannian groupoid $(\mathcal{G},M,g)$, let $\pi: F_{g} \rightarrow M$ be the orthonormal frame bundle associated to the metric $g$. This is a principal $O(n)$-bundle over $M$. There is a right action of $\mathcal{G}$ on $F_{g}$  defined by saying that if $h \in \mathcal{G}$ and $f$ is an orthonormal frame at $r(h)$ then $f\cdot h$ is the frame $(dh)^{-1}(f)$ at $s(h)$. We can then construct the corresponding cross-product groupoid with space of ``arrows''
\begin{align*}
\hat{\mathcal{G}}=F_{g} \rtimes \mathcal{G}=\{(f,h) \in F_{g} \times \mathcal{G}: \pi(f)=r(h)	\}
\end{align*}
and space of units $F_{g}$. The source and range maps are defined by $s(f,h)=f\cdot h$  and $r(f,h)=f$ respectively.  It is clear that $(\hat{\mathcal{G}},F_{g})$ has the structure of a smooth proper Lie groupoid and the structure of a smooth \'etale groupoid. The latter property implies, in particular, that the linearization $d\hat{h}:T_{s(\hat{h})}F_{g} \rightarrow T_{r(\hat{h})} F_{g}$ is well-defined for any $\hat{h} \in \hat{\mathcal{G}}$. Also, the groupoid $(\hat{\mathcal{G}},F_{g})$ has trivial isotropy groups.
       
			Let $\nabla$ be the Levi-Civita connection of the Riemannian metric $g$.  This connection defines an $O(n)$-invariant Riemannian metric $\hat{g}$ on $F_{g}$ such that the map $\pi : (F_{g},\hat{g}) \rightarrow (M,g)$ is a Riemanniann submersion and such that the restriction of $\hat{g}$ to a $\pi$-fiber is equal to the bi-invariant metric on $O(n)$ with unit volume. Since the metric $g$ on $M$ is $\mathcal{G}$-invariant, the metric $\hat{g}$ on $F_{g}$ is $\hat{\mathcal{G}}$-invariant.

  Let $Z$ be the space of orbits of the groupoid $(\hat{\mathcal{G}},F_{g})$ with quotient map $\hat{\sigma}: F_{g}\rightarrow Z$ . The connection $\nabla$ defines a canonical parallelism of $F_{g}$ which is invariant under the action of $\hat{\mathcal{G}}$.  It follows from \cite{salem1}[Theorem 4.2] that $Z$ is a smooth manifold and the quotient map $\hat{\sigma}$  is a smooth submersion. In fact, since the metric $\hat{g}$ on $F_{g}$ is $\hat{\mathcal{G}}$-invariant, it induces a metric $\overline{g}$ on $Z$ which makes $\hat{\sigma}$ a Riemannian submersion. Furthermore, the action of $O(n)$ on $F_{g}$ naturally induces an action of $O(n)$ on $Z$ and the metric $\overline{g}$ is invariant under this action. We have a commutative diagram 
\begin{eqnarray}
 F_{g}&\xrightarrow{\hat{\sigma}}& Z \notag\\
 \pi \downarrow &                & \downarrow \hat{\pi}  \label{CommutativeDiagram} \\
 M    &\xrightarrow{\sigma}      & W \notag
 \end{eqnarray}
 where the map $\hat{\sigma}$ is $O(n)$-equivariant and the maps $\pi$ and $\hat{\pi}$ correspond to taking $O(n)$- quotients. The space $Z$ is compact since $W$ is compact.

  The space $\hat{\mathcal{G}}$ comes from an equivalence relation on $F_{g}$ defined by saying that $f,f^{\prime}$ are equivalent if and only if $\hat{\sigma}(f)=\hat{\sigma}(f^{\prime})$. More precisely, there is an $O(n)$-equivariant isomorphism $\mathcal{G} \simeq F_{g} \times_{Z} F_{g}$.

		\subsection{ The Haar system generated by an invariant metric and the associated mean curvature form} \label{subsec:Haar-system}
		As we mentioned earlier, a Haar system for a proper Lie groupoid $(\mathcal{G},M)$ is a collection of positive measures $\{ d\rho^{p} \}_{p \in M}$ which is $\mathcal{G}$-invariant in the appropriate sense and such that the support of the measure $d\rho^{p}$ is $\mathcal{G}^{p}$ for any $ p \in M$. Furthermore, given any function $\alpha \in C_c(\mathcal{G})$, the function
	\begin{eqnarray*}
	 p \rightarrow \int_{\mathcal{G}^p} \alpha(h)\,d\rho^{p}(h)
	\end{eqnarray*}
	is in $C_c(M)$ (\cite{tu}[Definition 1.1]).

		Let $(\mathcal{G}, M,g)$ be a closed Riemannian groupoid as above. For $f \in F_{g}$, let $d\rho^{f}$ be the measure on $\hat{\mathcal{G}}$ which is supported on $\hat{\mathcal{G}}^{f} \simeq \hat{\sigma}^{-1}\left(\hat{\sigma}(f)\right)$ and is given there by the fiberwise Riemannian density. The family $\{d\rho^{f} \} _{f \in F_{g}}$  is a Haar system for the proper Lie groupoid $(\hat{\mathcal{G}},F_{g})$. In particular, the family of measures $\{d\rho^{f} \} _{f \in F_{g}}$ is $\hat{\mathcal{G}}$-invariant in an appropriate sense.
	
	Let $p \in M$ and $f \in F_{g}$ be such that $\pi(f)=p$. There is a diffeomorphism $i_{p,f}:\mathcal{G}^{p} \rightarrow \hat{\mathcal{G}}^{f}$ given by $i_{p,f}(h)=(f,h)$. Let $d\rho^{p}$ be the measure on $\mathcal{G}$ which is supported on $\mathcal{G}^{p}$ and is given there by the pullback measure $i_{p,f}^{\ast}d\rho^{f}$. Since the family of measures $\{d\rho^{f}\}_{f \in \pi^{-1}(p)}$ is $O(n)$-equivariant, the measure $d\rho^{p}$ is independent of the choice of $f$. 
	  
		The family of measures $\{d\rho^{p}\}_{p \in M}$ is a Haar system for the proper Lie groupoid $(\mathcal{G}, M)$. As follows from \cite{gorokhovsky}[Lemma 4] or \cite{tu2}[Proposition 6.11], we can construct a smooth nonnegative cutoff function for this Haar system. That is, there exists a nonnegative smooth function $ \phi \in C^{\infty}(M)$ such that
\begin{enumerate}
 \item For any compact subset $K$ of $M$, the set $supp(\phi) \cap s(r^{-1}(K))$ is also compact. 
 \item $\displaystyle \int_{\mathcal{G}^p} \phi^2(s(h)) \, d\rho^{p}(h)=1$  for any $p \in M$.
 \end{enumerate}
  Since we are assuming that the space of orbits $W$ is compact, the first condition simply means that $\phi$ has compact support.
	
Let $\hat{\theta}$ be the mean curvature form of the induced Riemannian submersion $\rho: (F_{g},\hat{g}) \rightarrow (Z,\overline{g})$. By \cite{gorokhovsky}[Lemma 3], the form $\hat{\theta}$ is a closed one-form which is $\hat{\mathcal{G}}$-basic and $O(n)$-basic. Let $\theta$ be the unique one-form on $M$ such that $\pi^{\ast} \theta =\hat{\theta}$. The form $\theta$ is closed and $\mathcal{G}$-invariant. 
         
				Let $E \rightarrow M$ be a vector bundle on $M$ equipped with a right $\mathcal{G}$-action and let $\nabla^{E}$  be a $\mathcal{G}$-invariant connection on $E$. Given $h \in \mathcal{G}$ and $e \in E_{s(h)}$, let $e \cdot h^{-1} \in E_{r(g)}$ denote the action of $h^{-1}$ on $e$. For any compactly supported element $\xi \in C^{\infty}_{c}(M;E)$,  we can construct the section  
\begin{eqnarray}
	\int_{\mathcal{G}^{p}} \xi_{s(h)}\cdot h^{-1} \, d\rho^{p}(h) \label{g-averaging}
\end{eqnarray}
  which will be a $\mathcal{G}$-invariant element of $C^{\infty}(M;E)$ whose value at $p \in M$ is given by (\ref{g-averaging}).  By \cite{gorokhovsky}[Lemma 6], we have the following identity in $\Omega^{1}(M;E)$:
	\begin{eqnarray*}
\nabla^{E}\int_{\mathcal{G}^{p}} \xi_{s(h)}\cdot h^{-1} \, d\rho^{p}(h)&=&\int_{\mathcal{G}^{p}} (\nabla^{E} \xi)_{s(h)}\cdot h^{-1} \, d\rho^{p}(h) \\
                                                                       & &  + \, \theta_{p} \int_{\mathcal{G}^{p}} \xi_{s(h)}\cdot h^{-1} \, d\rho^{p}(h).
\end{eqnarray*}
	
	It follows that, for any $\omega \in \Omega^{*}_{c}(M)$, we have 
 \begin{eqnarray}
d\int_{\mathcal{G}^{p}} \omega_{s(h)}\cdot h^{-1} \, d\rho^{p}(h)&=&\int_{\mathcal{G}^{p}} (d\omega)_{s(h)}\cdot h^{-1} \, d\rho^{p}(h)  \notag \\
                                                                 & & + \, \theta_{p} \wedge \int_{\mathcal{G}^{p}} \omega_{s(h)}\cdot h^{-1} \, d\rho^{p}(h). \label{differentiate}
\end{eqnarray}
In particular, if we take $\omega$ to be equal to the function $\phi^2$ and apply the definition of the nonnegative cutoff function, we get
\begin{eqnarray}
\theta_{p}=-\int_{\mathcal{G}^{p}} (d\phi^2)_{s(h)}\cdot h^{-1} \, d\rho^{p}(h) \label{cutoffprop2}
\end{eqnarray}
for every $p \in M$.  
 
\subsection{Integration of $\mathcal{G}$-invariant functions.} \label{subsec:integration-of-Gfunctions}
  We denote by $d\mu_{g}$ the Riemannian measure on $(M,g)$ and by $C_{\mathcal{G}}(M)$  the space of continuous $\mathcal{G}$-invariant functions on $M$. We will integrate  these functions by using the measure $\phi^2 d\mu_{g}$ on $M$. It is shown in \cite{gorokhovsky}[Proposition 1] that the integral
\begin{eqnarray*}
\int_{M} f  \, \phi^2 d\mu_{g}
\end{eqnarray*}
does not depend on the nonnegative cutoff function $\phi$ for any $f \in C_{\mathcal{G}}(M)$. More precisely, let $d\mu_{\overline{g}}$ be the Riemannian measure on $(Z,\overline{g})$ and let $d\eta(g)= \hat{\pi}_{\ast}d\mu_{\overline{g}}$ be the pushforward measure on $W$. A $\mathcal{G}$-invariant continuous function on $M$ can be viewed as a continuous function on $W$ and as an $O(n)$-invariant continuous function on $Z$. It follows from \cite{gorokhovsky}[Proposition 1] that the following identity holds
\begin{eqnarray*}
\int_{M} f \phi^2 \, d\mu_{g}= \int_{Z} f \, d\mu_{\overline{g}}= \int_{W} f \, d\eta(g)
\end{eqnarray*}
for any $f \in C_{\mathcal{G}}(M)$. Furthermore, the proof of \cite{gorokhovsky}[Proposition 1] can be adapted to obtain the following integration by parts formula.

\begin{lma} \label{byparts-lemma}
Let $\alpha$ be a $(r,s)$-tensor on $M$ and $\omega$  an $(r-1,s)$-tensor on $M$. If $\alpha$ and $\omega$ are $\mathcal{G}$-invariant, then 
\begin{eqnarray}
\int_{M} \langle \mathrm{div}(\alpha), \omega \rangle \, \phi^2 d\mu_{g} &=& -\int_{M} \langle \alpha, \nabla \omega\rangle \, \phi^2 d\mu_{g}  + \int_{M} \langle \iota_{\theta^{\sharp}} \alpha , \omega \rangle \, \phi^2 d\mu_{g} \label{integpart}
\end{eqnarray} 
                                                                         
where $\langle , \rangle$ denotes the inner product defined by the metric $g$ and $\iota_{\theta^{\sharp}} \alpha$ is the interior product of the dual $\theta^{\sharp}$ of $\theta$ with $\alpha$.
  \end{lma}
	\begin{proof}
	Since the nonnegative cutoff function $\phi$ has compact suppport, we can use the regular integration by parts formula.
	\begin{eqnarray*}
\int_{M} \langle \mathrm{div}(\alpha), \omega \rangle \, \phi^2 d\mu_{g} &=& -\int_{M}  \langle \alpha, \nabla(\phi^2 \omega) \rangle \,d\mu_{g}.
\end{eqnarray*}
                                                     
This gives
\begin{eqnarray*} 
\int_{M} \langle \mathrm{div}(\alpha), \omega \rangle \, \phi^2 d\mu_{g} &=& -\int_{M}  \langle \alpha, \nabla\omega \rangle  \, \phi^2 d\mu_{g} -\int_{M} \langle \alpha, \nabla \phi^2 \otimes \omega \rangle \, d\mu_{g}.
\end{eqnarray*}
                                                           
We  can rewrite the second integrand of the last equation as 
\begin{eqnarray*}
-\int_{M} \langle \alpha, \nabla \phi^2 \otimes \omega \rangle \, d\mu_{g}= -\int_{M} \langle \alpha \ast \omega, \nabla \phi^2 \rangle \, d\mu_{g}
\end{eqnarray*}
where $\alpha \ast \omega$ is the 1-form given by $\alpha \ast \omega(X)= \langle \iota_{X}\alpha,\omega \rangle$ for any vector $X$. Then, just as in the proof of \cite{gorokhovsky}[Proposition 1], we rewrite 
\begin{eqnarray}
-\int_{M} \langle \alpha \ast \omega, \nabla \phi^2 \rangle \, d\mu_{g}&=&-\int_{F_{g}} \pi^{\ast}(\langle \alpha \ast \omega,\nabla \phi^2 \rangle) \, d\mu_{\hat{g}}  \notag\\
                                                        &=& -\int_{Z} \left(\int_{F_{g}/Z} \pi^{\ast}(\langle \alpha \ast \omega,\nabla \phi^2 \rangle) d\mu_{\hat{g}}\right)d\mu_{\overline{g}}. \label{trick}
	\end{eqnarray}
The inner integral of the last equation corresponds to integrating over the ``fibers'' $\sigma^{-1}(\sigma(f)) \simeq \hat{\mathcal{G}}^{f}$ where $f \in F_{g}$. Note that the function $\hat{\phi}=\pi^{\ast}\phi$ is a cutoff function for the Haar system $\{ d\rho^{f} \}_{f \in F_{g}}$.  Also, just like Equation~(\ref{cutoffprop2}), we have 
	\begin{eqnarray*}
\hat{\theta}_{f}=-\int_{\hat{\mathcal{G}}^{f}} (d\hat{\phi}^2)_{s(\hat{h})}\cdot \hat{h}^{-1} \, d\rho^{f}(\hat{h})
\end{eqnarray*}
Hence, if $\chi$ is a $\hat{\mathcal{G}}$-invariant 1-form on $F_{g}$, we have
\begin{eqnarray*}
\langle\hat{\theta}_{f},\chi_{f} \rangle	&=& -\int_{\hat{\mathcal{G}}^{f}} \langle (\nabla\hat{\phi}^2)_{s(\hat{h})},\chi_{s(\hat{h})}\rangle \, d\rho^{f}(\hat{h})		\\
\int_{\hat{\mathcal{G}}^{f}} \langle \hat{\theta}_{s(\hat{h})},\chi_{s(\hat{h})} \rangle \, \hat{\phi}^2(s(\hat{h}))d\rho^{f}(\hat{h})&=&-\int_{\hat{\mathcal{G}}^{f}} \langle(\nabla\hat{\phi}^2)_{s(\hat{h})},\chi_{s(\hat{h})} \rangle \, d\rho^{f}(\hat{h}).
\end{eqnarray*}
If we apply this to (\ref{trick}), we obtain
\begin{eqnarray*}
	-\int_{M} \langle \alpha \ast \omega, \nabla \phi^2 \rangle \, d\mu_{g}&=&	\int_{Z} \left(\int_{F_{g}/Z} \langle \pi^{\ast}(\alpha \ast \omega),\hat{\theta} \rangle\, \hat{\phi}^2 d\mu_{\hat{g}}\right)d\mu_{\overline{g}}\\
	                            &=& \int_{M} \langle \alpha \ast \omega, \theta \rangle \, \phi^2 d\mu_{g} \\
															&=&  \int_{M} \langle \iota_{\theta^{\sharp}} \alpha , \omega \rangle \, \phi^2 d\mu_{g}.
\end{eqnarray*}	
So 
\begin{eqnarray*}
\int_{M} \langle \mathrm{div}(\alpha), \omega \rangle \, \phi^2 d\mu_{g} &=&-\int_{M}  \langle \alpha, \nabla\omega \rangle  \, \phi^2 d\mu_{g}  + \int_{M} \langle \iota_{\theta^{\sharp}} \alpha , \omega \rangle \, \phi^2 d\mu_{g}.
	\end{eqnarray*}													                                        
	This completes the proof.
	\end{proof}
	One can easily generalized this result to the case where the space of orbits $W$ is not necessarily compact. We will have to assume then that the image, under the quotient map $\sigma$, of the support of the tensor $\alpha$ or the tensor $\omega$ is compact.
	
	\subsection{The family of Haar systems $\Xi$} \label{subsec:evolution-formulas}
	Suppose $(\mathcal{G},M^{n})$ is a smooth complete effective \'etale groupoid and suppose that, for some $T < \infty$,  we have a smooth time-dependent family of $\mathcal{G}$-invariant Riemannian metrics $g(t)$ on $M$ defined for $t \in [0,T]$. The closed groupoid $(\overline{\mathcal{G}},M^{n})$ does not depend on the choice of the $\mathcal{G}$-invariant metric $g$. As before, we will assume that $\overline{\mathcal{G}}=\mathcal{G}$ and that the space of orbits $W$ is compact and connected. We will determine how the Haar system and the mean curvature form constructed in the previous subsections vary with respect to time. Given two $\mathcal{G}$-invariant metrics $g_1$ and $g_2$, one can see that, based on the way we  constructed the Haar system, there is some $\mathcal{G}$-invariant smooth function $w$ such that $d\rho^{p}_2 = e^w d\rho^{p}_1$ where $ \{d\rho^{p}_1 \}_{p \in M}$ and $\{d\rho^{p}_2\}_{p \in M}$ are the Haar systems generated by $g_1$ and $g_2$ respectively. So given a one parameter family of invariant metrics $g(t)$, there will be a time-dependent $\mathcal{G}$-invariant smooth function $\beta$ such that the evolution equation of the Haar system will be given by $\frac{\partial (d\rho^{p})}{\partial t} = \hat{\beta}(p)d\rho^{p}$. We will now show this precisely. The Haar system and the associated mean curvature form were obtained by looking at the action of the groupoid on the orthonormal frame bundle. This space depends on the choice of the metric $g$. To bypass this problem, we will use the ``Uhlenbeck trick''.
	
	Suppose that the evolution equation of the metric $g(t)$ is given by 
	\begin{eqnarray}
	\frac{\partial g}{\partial t} = v \label{g-evolution}
	\end{eqnarray}
	where $v=v(t)$ is a smooth $\mathcal{G}$-invariant symmetric 2-tensor on $M$. Let $E$ be a vector bunle over $M$ isomorphic to the tangent bundle $TM$ via a vector bundle isomorphism  $u_{0}: E \rightarrow  TM$. We equip $E$ with the bundle metric $(.,.)=u_{0}^{\ast}g_{0}$ where $g_{0}=g(0)$. The map $u_{0}$ is now a bundle isometry.
	     
			For $t \in [0,T]$, consider the one-parameter family of bundle isomorphisms $u_{t}: E \rightarrow M$ that satisfy the ODE
	\begin{eqnarray}
	\frac{\partial u}{\partial t}&=& -\frac{v}{2}\circ u  \label{u-trick}\\
	u(0)&=& u_{0} \notag
	\end{eqnarray}
		where we use the metric $g(t)$ to view $v(t)$ as a bundle map $TM \rightarrow TM$. This ODE has a solution defined on $[0,T]$. Furthermore, by construction, the map $u_{t}$ is a bundle isometry $(E,(.,)) \rightarrow (T_{g(t)}M,g_t)$ for each $t$.
		
		The right action of $\mathcal{G}$ on $TM$ induces, for each $t \in [0,T]$, a right action of $\mathcal{G}$ on $E$ given by 
\begin{eqnarray}
w\cdot h= u_{t}^{-1}\left((u_{t}(w))\cdot h\right)=u_{t}^{-1}\circ (dh)^{-1} \circ u_{t}(w) \label{induced-action}
\end{eqnarray}

where $h \in \mathcal{G}$ and $w \in E_{r(h)}$.

\begin{lma}
The right  action of $\mathcal{G}$ on $E$ does not depend on $t$.
\end{lma}
\begin{proof}
If we differentiate Equation~(\ref{induced-action}) with respect to time, we get
\begin{eqnarray*}
\frac{\partial (w\cdot h)}{\partial t}=  \frac{\partial u^{-1}}{\partial t}\circ (dh)^{-1} \circ u(w) + u^{-1}\circ (dh)^{-1} \circ \frac{\partial u}{\partial t}(w). \label{proof2}
\end{eqnarray*}
Using (\ref{u-trick}), we get
\begin{eqnarray*}
\frac{\partial u^{-1}}{\partial t}= u^{-1} \circ \frac{v}{2}.
\end{eqnarray*}
Hence, Equation~(\ref{proof2}) becomes 
\begin{eqnarray*}
\frac{\partial (w\cdot h)}{\partial t}= u^{-1} \circ \frac{v}{2} \circ (dh)^{-1} \circ u(w) -u^{-1}\circ (dh)^{-1} \circ \frac{v}{2}\circ u(w).
\end{eqnarray*}
Since the tensor $v$ is $\mathcal{G}$-invariant, it follows that
\begin{eqnarray*}
\frac{v}{2} \circ (dh)^{-1} = (dh)^{-1} \circ \frac{v}{2} .
\end{eqnarray*}
Therefore,
\begin{eqnarray*}
\frac{\partial (w\cdot h)}{\partial t}=0.
\end{eqnarray*}
This completes the proof.
\end{proof}

Let $ \pi: \mathbf{F} \rightarrow  M$ be the orthonormal frame bundle of the metric bundle $(E,h) \rightarrow M$. For each $t \in [0,T]$, the map $u_{t}$ induces a bundle isomorphism $\mathbf{F} \rightarrow F_{g(t)}$. Also, it follows from the previous lemma that the induced cross-product groupoid over $\mathbf{F}$ with $\hat{\mathcal{G}}=\mathbf{F} \rtimes \mathcal{G}$ does not depend on $t$.

   Just as in subsection~\ref{subsec:induced-groupoid}, the space of orbits $Z=\mathbf{F}/\hat{\mathcal{G}}$ is a smooth manifold and the quotient map $\hat{\sigma}: \mathbf{F} \rightarrow Z$ is a smooth submersion. Indeed, for each $t \in [0,T]$, the pullback connection $A(t)=u_{t}^{\ast}\nabla_{g(t)}$ induces a canonical parallelism on $\mathbf{F}$ which is preserved by $\hat{\mathcal{G}}$. Furthermore, we get a smooth time-dependent $\hat{\mathcal{G}}$-invariant and $O(n)$-invariant metric $\hat{g}(t)$ on $\mathbf{F}$ which induces a smooth time dependent $O(n)$-invariant Riemannian metric $\overline{g}(t)$ on $Z$. This makes $\hat{\sigma}: (\mathbf{F},\hat{g}(t)) \rightarrow (Z,\overline{g}(t))$ an $O(n)$-equivariant Riemannian submersion for each $t$. We can then construct the time dependent Haar system $\{d\rho^{f} \} _{f \in \mathbf{F}}$ and mean curvature form $\hat{\theta}$ as described in subsection~\ref{subsec:Haar-system}. For each $t$, this will induce the Haar system $\{d\rho^{p}\}_{p \in M}$ and the mean curvature form $\theta$ corresponding to $g(t)$. We can then construct a smooth time-dependent nonnegative cutoff function $\phi$ for the Haar system.
	
	The evolution equation for the Riemannian metric $\hat{g}$ on $\mathbf{F}$ has the form
	\begin{eqnarray}
	\frac{\partial \hat{g}}{\partial t} = \hat{v} \label{Fg-evolution}
	\end{eqnarray}
where $\hat{v}$ is a $\hat{\mathcal{G}}$ and $O(n)$-invariant symmetric 2-tensor on $\mathbf{F}$. Recall that, for each $f \in \mathbf{F}$ and each $t \in [0,T]$, the measure $d\rho^{f}$ on $\hat{\mathcal{G}}^{f} \simeq \hat{\sigma}^{-1}(\sigma(f))$ is the fiberwise Riemannian density with respect to the metric $\hat{g}$. So the evolution equation for the measure $ d\rho^{f}$ is given by
	\begin{eqnarray}
\frac{\partial (d\rho^{f})}{\partial t} = \hat{\beta}d\rho^{f}. \label{haar-evolution1}
\end{eqnarray}
  where $2 \hat{\beta}$ is equal to the trace of $\hat{v}$ along the fiber $\hat{\sigma}^{-1}(\sigma(f))$ with respect to the metric $\hat{g}$. The function $\hat{\beta}$ is a $\hat{\mathcal{G}}$ and $O(n)$-invariant smooth function on $\mathbf{F}$. It induces a $\mathcal{G}$-invariant function $\beta$ on $M$.
	
	\begin{lma}
	The evolution equations for the Haar system $\{d\rho^{p}\}_{p \in M}$ and the mean curvature form $\theta$ are given by 
\begin{eqnarray}
\frac{\partial (d\rho^{p})}{\partial t} = \beta(p)d\rho^{p}. \label{haar-evolution2}
\end{eqnarray}
for every $p \in M$ and 
\begin{eqnarray}
\frac{\partial \theta}{\partial t}= d\beta \label{theta-evolution}.
\end{eqnarray}
	\end{lma}
	\begin{proof}
		The first equation follows easily from (\ref{haar-evolution1}) and the fact that $\beta$ is $\mathcal{G}$-invariant. To prove the second equation, we start by writing the Haar system as 
\begin{eqnarray*}
 d\rho^{p}_{t}= e^{w(p,t)} d\rho^{p}_{0}
\end{eqnarray*}
where $\{  d\rho^{p}_{0} \}_{p \in M} $ is the Haar system at time $t=0$ and where the $\mathcal{G}$-invariant function $w$ is defined by $w(p,t) = \int_{0}^{t} \beta(p,s) \, ds$. It follows that a nonnegative cutoff function for time $t$ is given by $ \phi^2_t = e^{-w} \phi^2_{0}$ where $\phi^2_{0}$ is the nonnegative cutoff function at time $t=0$.  Then, using Equation~(\ref{cutoffprop2}), we get
\begin{eqnarray*}
 \theta = dw + \theta_0
\end{eqnarray*}
where $\theta_0$ is the mean curvature form at time $t=0$. Differentiating with respect to time gives us equation (\ref{theta-evolution}).
\end{proof}

So, as we stated earlier, the smooth Haar system constructed in Subsection~\ref{subsec:Haar-system} is independent of the $\mathcal{G}$-invariant metric $g$ modulo scaling by smooth positive $\mathcal{G}$-invariant functions. We will denote the family of these Haar systems by $\Xi$: an element of $\Xi$ is a Haar system that can be written in the form $\{e^{w}d\rho^{p}\}_{p \in M}$ where $w$ is a $\mathcal{G}$-invariant function and where $\{d\rho^{p}\}_{p \in M}$ is the Haar system generated by some  $\mathcal{G}$-invariant Riemannian metric $g$.
In particular, if we fix an element $d\rho_0$ of $\Xi$ and if we denote by $C^{\infty}_{\mathcal{G}}(M)$ the space of $\mathcal{G}$-invariant smooth functions on $M$, we get a bijection $C^{\infty}_{\mathcal{G}}(M) \rightarrow  \Xi$ defined by $w \rightarrow e^{w} d\rho$.

	It also follows from the previous Lemma that the $\mathcal{G}$-invariant de Rham cohomology class of the mean curvature form $\theta$ is independent of the $\mathcal{G}$-invariant metric $g$. We will denote this class by $\Theta$. Furthermore, one can see that if a representative $\theta_1$ of $\Theta$ is written in the form $\theta_1= \theta_2  + dw $ where $\theta_2$ is the mean curvature form associated to a Haar system $\{d\rho_2 \}_{p \in M}$ generated by a metric $g$, then $\theta_1$ corresponds to the Haar system $\{d\rho_1 \}_{p \in M}=\{e^{w}d\rho_2 \}_{p \in M}$ which can be thought as coming from a $\mathcal{G}$-invariant and $O(n)$-invariant rescaling of the metric $\hat{g}$ along the $\hat{\sigma}$-fibers of $F_{g}$. A nonnegative cutoff function for the Haar system $\{d\rho_1 \}_{p \in M}$ will be said to be compatible with $\theta_1$ and it is clear that, if $\phi_2$ is a nonnegative cutoff function for $\{d\rho_2 \}_{p \in M}$, then a nonnegative cutoff function for $\{d\rho_1 \}_{p \in M}$ is given by $\phi_1^2 =e^{-w} \phi_2^2$. One can then easily deduce the following result.

\begin{crly} \label{byparts-cr2}
  If $g$ is a $\mathcal{G}$-invariant metric and if $\theta$ is any representative of the class $\Theta$, then the integration by parts formula of Lemma~\ref{byparts-lemma} is valid for any nonnegative cutoff function $\phi$ that is compatible with $\theta$.
	\end{crly}
	
	We can also adapt the proof of Lemma~\ref{byparts-lemma} to show that, given a smooth one-parameter family of $\mathcal{G}$-invariant metrics $g_t$, a smooth time dependent $\mathcal{G}$-invariant function $f$, a smooth family of Haar systems $\{d\rho_t \}_{p \in M}$ in $\Xi$ and a smooth compatible one-parameter family of nonnegative cutoff functions $\phi_t$, the evolution equation of the integral $ \int_{M} f \phi^2 d\mu_{g}$ is given by
	                                                                     
\begin{eqnarray}
\frac{\partial}{\partial t} \int_{M} f \phi^2 d\mu_{g} =  \int_{M} \left( \frac{\partial f}{\partial t} + \frac{V}{2} -\beta \right) \, \phi^2 d\mu_{g} \label{time-derivative}
\end{eqnarray} 
 where $V= \mathrm{Tr}_{g}(v)$, the trace of $v$ with respect to the metric $g$, and 
\begin{eqnarray*}
\frac{\partial (d\rho^{p})}{\partial t} = \beta(p)d\rho^{p}.
\end{eqnarray*}

\begin{exam} \label{groupExample}
Let $G$ be a a compact Lie group which acts isometrically and effectively on a connected compact Riemannian manifold $(M,g)$. Then, we can consider the cross-product groupoid $ M \rtimes G$ over $M$ which will be a closed groupoid. The corresponding \'etale structure comes from equipping $G$ with the discrete topology. 

As discussed in \cite{gorokhovsky}[Example 4], the Haar system $\{d \rho^{p} \}_{p \in M}$ genereated by the metric $g$ can be described as follows. Let $\{e^i \}$ be a basis of the Lie algebra $\mathfrak{g}$ of $G$ such that the normalized Haar measure is $d\mu_{G}= \bigwedge_{i}(e^i)^{\ast}$. Let $\{V^i\}$ be the corresponding Killing vector fields on $M$. The action of $V^i$ on the orthonormal frame bundle $F_{g}$ breaks up as $V^i \oplus \nabla V^{i}$ with respect to the decomposition $TF_{g}= \pi^{\ast} TM \otimes TO(n)$ of $TF_{g}$ into horizontal and vertical subbundles. The Haar system is given by $d\rho^p = \sqrt{\mathrm{det}(M(p))}d\mu_{G} $ where $M=(M_{ij})$ is a function of positive definite matrices defined by 
\begin{eqnarray*}
M_{ij}(p)=  \langle V_i(p), V_j(p) \rangle +  \langle \nabla V_i(p), \nabla V_j(p) \rangle
\end{eqnarray*}
The corresponding cutoff function is given by $\phi^2(p) = \frac{1}{\sqrt{\mathrm{det}(M(p))}}$ and the mean curvature form is $\theta= dw$ with $w = \ln\left(\sqrt{\mathrm{det}(M(p))}\right)$. Hence, the class $\Theta$ is equal to zero. Also, the family $\Xi$ of Haar systems can be naturally identified with the space of $G$-invariant smooth functionns via the map $C^{\infty}_{G}(M) \rightarrow  \Xi$ defined by $w \rightarrow e^{w} d\rho_{G}$.
Finally, Corollary~\ref{byparts-cr2} reduces to the integration by parts formula on a weighted manifold with $G$-invariant quantities.
\end{exam}

\subsection{$\mathcal{G}$-paths and $\mathcal{G}$-invariant vector fields}\label{subsec:Gvectorfields}
 As before, we let $(\mathcal{G},M,g)$ be a closed Riemannian groupoid with space of orbits $W=M/ \mathcal{G}$. A \emph{smooth} $\mathcal{G}$-\emph{path} defined on an interval $[a,b]$ and from a point $x \in M$ to a point $y \in M$ consists of a partition $a=t_o \leq t_1 \leq  \cdots \leq t_k =b$ of the interval $[a,b]$ and a sequence $c= (h_0,c_1,h_1,\cdots,c_k,h_k)$ where each $c_i$ is a smooth path $c_i: [t_{i-1},t_i] \rightarrow M$, $h_i \in \mathcal{G}$, $c_i(t_{i-1})=s(h_{i-1})$, $c_i(t_i)=r(h_i)$, $r(h_0)=x$ and $s(h_k)=y$. The groupoid is said to be $\mathcal{G}$-\emph{path-connected} if any two points of $M$ can be joined by a $\mathcal{G}$-path. 
      
			Since $M$ is equipped with the $\overline{\mathcal{G}}$-invariant metric $g$, there is a natural notion of length of a smooth $\overline{\mathcal{G}}$-path. This defines a pseudometric $d$ on the space of orbits $W$ given by saying that, for any $x,y \in M$, $d(O_{x},O_{y})$ is the infimum of the lengths of smooth $\overline{\mathcal{G}}$-paths joining $x$ to $y$. We will also assume for now on that the groupoid $(\mathcal{G},M)$ is \emph{Hausdorff} in the sense of \cite[Appendix G Definition 2.10]{bridson}: the space $\mathcal{G}$ equipped with the \'etale topology is Hausdorff and for every continuous map $c:(0,1]  \rightarrow \mathcal{G}$ such that $\lim_{t \rightarrow 0} s\circ c$ and $ \lim_{t \rightarrow 0} r\circ c$ exist, $ \lim_{t \rightarrow 0}  c$ exists. Using this property, we can prove the following lemma

\begin{lma}
 The  pseudometric $d$ on $W$ is a metric. Furthermore, the topology on $W$ induced by $d$ coincides with the quotient topology. Hence, $(W,d)$ is a length space.
 \end{lma}
 \begin{proof}
 Suppose that $x$ and $y$ are elements of $M$ such that $d(O_{x},O_{y})=0$. Let $\epsilon >0$ be small enough so that the closed ball $\overline{B}(x,\epsilon)$ is complete as a metric space. It follows from \cite[Appendix G Lemma 2.14]{bridson} that $\overline{B}(x,\epsilon)$ will contain an element of $O_{y}$. We can then construct a sequence $y_i \in O_{y}$ such that $y_i \rightarrow x$. Since $O_{y}$ is a closed subset of $M$, we must have $x \in O_{y}$. Hence, $O_{x}=O_{y}$. The pseudometric $d$ is indeed a metric.

 Now, suppose that $x \in M$ with orbit $O_x \in W$. Let $B(O_x,b)$ be a ball of radius $b$ with respect to the metric $d$ on $W$. We want to show that $B(O_x,b)$ is open in $W$ with respect to the quotient topology or, equivalently, that $\sigma^{-1}(B(O_x,b))$ is open in $M$. If $y \in \sigma^{-1}(B(O_x,b))$, then, by definition of the metric $d$, there exists a smooth $\mathcal{G}$-path $c$ of length less than $b$ joining $x$ to $y$. Then, it is clear that, for $\epsilon >0$ small enough, the ball $B(y,\epsilon)$ is a subset of  $\sigma^{-1}(B(O_x,b))$ since for any $q \in B(y,\epsilon)$, we can join $x$ to $q$ by a smooth $\mathcal{G}$-path of length less than $b$ obtained by ``concatenating'' the smooth $\mathcal{G}$-path $c$ with a geodesic joining $y$ to $q$. This proves that $\sigma^{-1}(B(O_x,b))$ is open in $M$ and hence $B(O_x,b)$ is open in $W$ with respect to the quotient topology.
    
    Now suppose $U \subset W$ is open with respect to the quotient topology. Then $\sigma^{-1}(U)$ is open in $M$. Let $x$ be an element of $\sigma^{-1}(U)$ and let $\epsilon >0$ be small enough so that $\overline{B}(x,\epsilon)$ is complete and contained in $\sigma^{-1}(U)$. Then, it follows again from \cite[Appendix G Lemma 2.14]{bridson} that every orbit in $B(O_{x},\epsilon)$ intersects the ball $B(x,\epsilon)$. In other words, $B(O_{x},\epsilon)=\sigma\left(B(x,\epsilon)\right) \subset U$. Hence, $U$ is open in $W$ with respect to the metric topology. The metric topology coincides with the quotient topology. Finally, it follows from \cite[Part I Lemma 5.20]{bridson} that $(W,d)$ is a length space. This completes the proof.
 \end{proof}

\begin{crly}
 $(\mathcal{G},M,g)$ is $\mathcal{G}$-path-connected if and only if the quotient space $W$ is connected.
 \end{crly}
 
Let $X$ be a $\mathcal{G}$-invariant vector field on the manifold $M$. Hence, under the right action of $\mathcal{G}$ on $TM$, we have $X_{r(h)} \cdot h = X_{s(h)}$. We can assign to such a vector field a collection of smooth $\mathcal{G}$-paths defined up to equivalence. Recall that two smooth $\mathcal{G}$-paths $c_1$ and $c_2$ defined on the same interval $[a,b]$ are said to be equivalent if one can pass from one to the other by performing the following operations:
\begin{enumerate}
\item Suppose $c= (h_0,c_1,h_1,\cdots,c_k,h_k)$ is a $\mathcal{G}$-path with the subdivision $a=t_o \leq t_1 \leq  \cdots \leq t_k =b$. We get a new $\mathcal{G}$-path by adding a new  subdivison point $t_i^{\prime} \in [t_{i-1},t_i]$ together with the unit element $g_{i}^{\prime}=1_{c_i(t_i^{\prime})}$. We replace $c_i$ by $c_i^{\prime}$, $g_i^{\prime}$ and $c_i^{\prime \prime}$, where $c_i^{\prime}$ and $c_i^{\prime \prime}$ are the restrictions of $c_i$ to $[t_{i-1},t_i^{\prime}]$ and $[t_i^{\prime},t_i]$ respectively.
\item We replace a $\mathcal{G}$-path $c$ by a new path $c^{\prime}=(h_0^{\prime},c_1^{\prime},h_1^{\prime},\cdots,c_k^{\prime},h_k^{\prime})$ over the same division as follows: for each $i =1, \cdots,k$ , choose a smooth path  $\gamma_{i}:[t_{i-1},t_{i}] \rightarrow \mathcal{G}$ with respect to the \'etale structure on $(\mathcal{G},M)$ such that $s(\gamma_{i}(t))=c_{i}(t)$ and define $c_{i}^{\prime}(t)=r(\gamma_{i}(t))$, $h_{i}^{\prime}=\gamma_{i}(t_i)h_i\gamma_{i+1}(t_i)^{-1}$ for $i =1, \cdots,k-1$ , $h_{0}^{\prime}=h_{0}h_{1}^{-1}(t_{0})$ and $h_{k}^{\prime}= \gamma_{k}(t_{k})h_{k}.$
\end{enumerate}

 The smooth $\mathcal{G}$-paths defined by a $\mathcal{G}$-invariant vector field $X$ will be of the form $c= (h_0,c_1,h_1,\cdots,c_k,h_k)$ where, for each $i=1,\cdots,k$, the smooth path $c_i:[t_{i-1},t_{i}] \rightarrow M$ is an integral curve of the vector field $X$. It is also clear that if $\alpha_{i} : U \rightarrow M$ is a local section of $s$ at $h_i$, then, for some small enough $\epsilon >0$, the path $c_{i}^{\prime}: [t_{i-1},t_{i}+\epsilon] \rightarrow M$ given by
 \begin{eqnarray*}
 c_{i}^{\prime}(t)=\begin{cases}
c_{i}(t) &  \,\text{if} \, \,t_{i-1} \leq t\leq t_i; \\
 r \circ \alpha_{i} \circ c_{i+1} (t), & \, \text{if} \, \, t_i \leq t \leq t_{i} + \epsilon.
\end{cases} 
 \end{eqnarray*}
 is an integral curve of the vector field $X$. A $\mathcal{G}$-path $c$ generated by $X$ and defined on an interval $[a,b]$ will induce a continuous path $\tilde{c}:[a,b] \rightarrow W$. Any other smooth $\mathcal{G}$-path generated by $X$ and defined on the same interval $[a,b]$ will induce the same path in $W$ if and only if it is equivalent to $c$. With abuse of terminology, we will also call these induced paths the integral curves of the vector field $X$ on $W$ and the image of these curves will be called the flow lines of $X$ on $W$. It is clear that there is a unique flow line passing through any element of $W$. The following lemma is a simple generalization of the ``Escape Lemma'' for vector fields on smooth manifolds.

\begin{lma}
Suppose $\gamma$ is an integral curve of $X$ on $W$. If the maximal domain of definition of $\gamma$ is not all of $\mathbb{R}$, then the image of $\gamma$ cannot lie in any compact subset of $W$.
\end{lma}
 \begin{proof}
 Let $(a,b)$ denote the maximal domain of $\gamma$ and assume that $b < \infty$ but $\gamma(a,b)$ lies in a compact set $K \subset W$. We will show that $\gamma$ can be extended past $b$, thus contradicting the maximality of $(a,b)$. The case $a > -\infty$ is similar. Given any sequence of times $\{t_j\}$ such that $t_{j} \nearrow b$, the sequence $\{\gamma(t_j)\}$ lies in $K$. After passing to a subsequence, we may assume that $\{\gamma(t_j)\}$ converges to the orbit $O_p$ for some $p \in M$. By \cite[Appendix G Lemma 2.14]{bridson} , we can choose $\epsilon >0$ small enough so that the closed ball $\overline{B}(p,\epsilon)$ is complete and such that the points $\gamma(t_j)$ will have a representative in $B(p,\frac{\epsilon}{2})$  for $j$ large enough.
   
	Since $X$ is $\mathcal{G}$-invariant, we can view $|X|=|X|_{g}$ as a function on $W$. Since $K$ is compact, we will have $|X| \leq C$ on $\sigma^{-1}(K)$ for some positive constant $C$. So, if $c= (h_0,c_1,h_1,\cdots,c_k,h_k)$ is a $\mathcal{G}$-path representing $\gamma$ on any subinterval $[d,e]$ of $(a,b)$, we will have $|\dot{c_i}(t)|\leq C$ for any $t \in [t_{i-1},t_i]$. Let $j_{0}$ large enough so that the point  $\gamma(t_{j_{0}})$ has a representative in $B(p,\frac{\epsilon}{2})$ and such that $C\cdot(b-t_{j_{0}}) \leq \frac{\epsilon}{2}$. We now choose a smooth $\mathcal{G}$-path $c= (h_0,c_1,h_1,\cdots,c_k,h_k)$ representing $\gamma$ on the interval $[0,t_{j_{0}}]$. There exists a point $q$ in $O_{c_{k}(t_{j_{0}})}$ which lies in $B(p,\frac{\epsilon}{2})$. Let $c^{\prime}$ be the integral curve of $X$ such that $c^{\prime}(t_{j_{0}})=q$. This integral curve must be defined on an interval bigger than $[t_{j_{0}},b)$. Otherwise, we would have $\mathrm{Length}(c^{\prime}) \leq C \cdot(b-t_{j_{0}}) \leq \frac{\epsilon}{2}$  and thus $c^{\prime}([t_{j_{0}},b))$ would lie in $\overline{B}(p,\epsilon)$, thus contradicting the ``Escape Lemma''. So, $c^{\prime}$ is defined on $[t_{j_{0}},b+ \epsilon^{\prime})$ for some $\epsilon^{\prime} >0$. We can then extend the smooth $\mathcal{G}$-path $c$ by ``concatenating'' it with the path $c^{\prime}$. This proves that $\gamma$ can be extended past $b$, contradicting the maximality of $(a,b)$.
\end{proof}

\begin{crly}
If the space of orbits $W$ is compact, the integral curves of $X$ on $W$ are defined for all time. In fact, for any $p \in M$ and $t \in \mathbb{R}$, there exists a smooth $\mathcal{G}$-path $c= (h_0,c_1,h_1,\cdots,c_k,h_k)$ generated by $X$, defined on the interval $[0,t]$ if $t >0$ (or $[t,0]$ if $t < 0$) and satisfying $r(h_{0})=p$ (or $s(h_{k})=p$).
\end{crly}
 Assume that $W$ is compact and connected. We will show that the $\mathcal{G}$-invariant vector field $X$ generates a self-equivalence for each $t \in \mathbb{R}$. First, we recall the notion. Let $\mathbf{H}$ be the pseudogroup of an effective \'etale groupoid $(\mathcal{G},M)$. So the elements of $\mathbf{H}$  are the locally defined diffeomorphisms of the form $r\circ \alpha$ where $\alpha$ is a local section of $s$. A differentiable equivalence between two effective \'etale groupoids $(\mathcal{G}_1,M_1)$ and $(\mathcal{G}_2,M_2)$ is a maximal collection $\Psi =\{\psi_{i}:U_{i} \rightarrow V_{i} \}_{i \in \Lambda}$ of diffeomorphisms of open subsets of $M_1$ to open subsets of $M_2$ such that:
\begin{enumerate}
\item $\{U_{i}\}_{i \in \Lambda}$ and $\{V_{i}\}_{i \in \Lambda }$ are open coverings of $M_1$ and $M_2$ respectively.
\item If $\psi \in \Psi, \tau_1 \in \mathbf{H}_1$ and $\tau_2 \in \mathbf{H}_2$, then $\tau_2 \circ \psi  \circ \tau_1 \in \Psi$.
\item If $\psi,\psi^{\prime} \in \Psi, \tau_1 \in \mathbf{H}_1$ and $\tau_2 \in \mathbf{H}_2$, then $\psi^{\prime} \circ \tau_1 \circ \psi^{-1} \in \mathbf{H}_2$ and $\psi^{-1} \circ \tau_2 \circ \psi^{\prime} \in \mathbf{H}_1$.
\end{enumerate} 
It is clear that a differentiable equivalence induces a weak equivalence in the sense of \cite{moerdijk}[Chapter 5.4]. If $(\mathcal{G}_1,M_1)=(\mathcal{G}_2,M_2)$, we say that $\Psi$ is a self-equivalence.
       Now, suppose that $(\mathcal{G},M,g)$ is a closed Riemannian groupoid with compact connected orbit space and let $\Psi=\{\psi\}_{i \in \Lambda}$ be a self-equivalence. We can define the pull back $\Psi^{\ast}\omega$ of a $\mathcal{G}$-invariant tensor $\omega$ as follows. For any $x \in M$, choose a $\psi_{i} \in \Psi$ such that $x$ is in the domain of $\psi_{i}$. We set $(\Psi^{\ast}\omega)_{x}=(\psi_{i}^{\ast}\omega)_{x}$. This does not depend on the choice of $\psi_{i}$. Furthermore, the tensor $\Psi^{\ast}\omega$ is $\mathcal{G}$-invariant. Applying this to the $\mathcal{G}$-invariant Riemannian metric $g$, we obtain another $\mathcal{G}$-invariant Riemannian metric $\Psi^{\ast}g=\tilde{g}$. If we denote by $d$ and $\tilde{d}$ the metrics generated by $g$ and $\tilde{g}$ on $W$, we see that $\Psi$ induces an isometry $(W,\tilde{d}) \rightarrow (W,d)$ which we will also call $\Psi$ for simplicity. This isometry satisfies $\Psi^{\ast} d\eta(g)=d\eta(\tilde{g})$. Indeed, note that the self equivalence $\Psi$ induces a differentiable equivalence between the groupoids $(F_{\tilde{g}} \rtimes \mathcal{G}, F_{\tilde{g}})$ and $(F_{g} \rtimes \mathcal{G}, F_{g})$ where $F_{\tilde{g}}$ and $F_{g}$ are the orthonormal frame bundles generated by $\tilde{g}$ and $g$ respectively. This differentiable equivalence will pull back the Riemannian metric on $F_{g}$ to the Riemannian metric on $F_{\tilde{g}}$ and it will be ``equivariant'' with respect to the $O(n)$-action on these spaces. Furthermore, it will induce an isometry $\tilde{Z} \rightarrow Z$ where $\tilde{Z}$ and $Z$ are the spaces of orbits of $F_{\tilde{g}}$ and $F_{g}$ under the groupoid action. Now, if we recall the definition of the measure $d\eta(g)$, we see that $\Psi^{\ast} d\eta(g)=d\eta(\tilde{g})$.
			
			Finally, based on the definition of the family $\Xi$, we see that we can define the pullback of an element of $\Xi$ under a self-equivalence $\Psi$. More precisely, if we write an element $\rho$ of $\Xi$ as $d\rho =e^{w}d\rho_{g}$ where $w$ is a smooth $\mathcal{G}$ invariant function and where $d\rho_{g}$ is the Haar system generated by an invariant metric $g$, then $\Psi^{\ast}d\rho= e^{\Psi^{\ast}w}d\rho_{\Psi^{\ast}g}$. It is clear that, if $\theta$ is the mean curvature form associated to $d\rho$, then $\Psi^{\ast}\theta$ is the mean curvature form for $\Psi^{\ast}d\rho$. We can also see that, if $\phi$ is any nonnegative cutoff function for a Haar system $\rho \in \Xi$ and if $\underline{\phi}$ is any nonnegative cutoff function for the pullback $\Psi^{\ast}d\rho$, then
\begin{eqnarray*}
	\int_{M} f \, \phi^2 d\mu_{g} = \int_{M} \Psi^{\ast}f  \,\, \underline{\phi}^2 d\mu_{\Psi^{\ast}g}
\end{eqnarray*}
for any $\mathcal{G}$-invariant smooth function $f$ and any $\mathcal{G}$-invariant smooth metric $g$.

			Given a $\mathcal{G}$-invariant vector field $X$, we get a one-parameter family of self equivalences in an obvious way. First, for each $t \in \mathbb{R}$, we denote by $\gamma_{t}$ the flow map of $X$ on $M$ at time $t$. More precisely, a point $p \in M$ is in the domain of $\gamma_{t}$ if there is, in the regular sense, an integral curve $c$ of $X$ on $M$ defined on $[0,t]$ (assuming $t\geq 0$) and starting at $p$. We then take $\gamma_{t}(p)=c(t)$. The domain of $\gamma_{t}$ is an open subset of $M$. Now  suppose $c= (h_0,c_1,h_1,\cdots,c_k,h_k)$ is a smooth $\mathcal{G}$-path generated by $X$, defined on the interval $[0,t]$, starting at a point $p \in M$ and with the subdivision $0=t_o \leq t_1 \leq  \cdots \leq t_k =t$ . For each $i$, we let $\tau_{i}$ be the diffeomorphisms of open subsets given by $r\circ \alpha_{i}$ where $\alpha_{i}$ is a local section of $s$ at $h_{i}$. We then define a diffeomorphism $\psi_{c}$ on an open neighborhood of $p$ by
	\begin{eqnarray*}
	\psi_{c}=	\tau_{0}^{-1} \circ \gamma_{t_1} \circ \tau_{1}^{-1} \circ \gamma_{(t_2-t_1)} \circ \cdots \tau_{k-1}^{-1}\circ\gamma_{(t-t_{k-1})}\circ \tau_{k}^{-1}.
	\end{eqnarray*}
	The diffeomorphism $\psi_{c}$ is just a composition of flow maps of the vector field $X$ and of local diffeomorphisms generated by the groupoid. It is clear that $\psi_{c}(p)=s(h_{k})$. The collection of all such locally defined diffeomorphism $\psi_{c}$ defines a self equivalence $\Psi_{t}$. It is also clear that $\Psi_{0}=\mathbf{H}$ and, for any $t_1,t_2 \in \mathbb{R}$, we have $\Psi_{t_1+t_2}= \Psi_{t_1} \circ \Psi_{t_2}$ where  $\Psi_{t_1} \circ \Psi_{t_2}$ denotes the collection of locally defined diffeomorphisms of the form $\psi_{c} \circ \psi_{c^{\prime}}$ where $\psi_{c} \in \Psi_{t_1}$ and $\psi_{c^{\prime}} \in \Psi_{t_2}$. We can easily show that the formula 
	\begin{eqnarray*}
	\frac{d}{dt}\arrowvert_{t=t_0} \Psi_{t}^{\ast}\omega= \Psi_{t_0}^{\ast}\left(\mathcal{L}_{X}\omega\right)
	\end{eqnarray*}
holds for any $\mathcal{G}$-invariant tensor $\omega$. Finally, note that the results listed in this section can easily be extended to the case of smooth time dependent $\mathcal{G}$-invariant vector fields $X=X(t)$. 

\subsection{$\mathcal{G}$-geodesics}
 A $\mathcal{G}$-geodesic on a Riemannian groupoid $(\mathcal{G},M,g)$ is a $\mathcal{G}$-path $c=(h_0,c_1,h_1, \cdots,c_k,h_k)$ over a subdivision $ a=t_{0} < \cdots < t_{k} =b$ of an interval $[a,b]$ satisfying the following conditions for $i=1,\cdots,k$.
\begin{enumerate}
\item $c_{i}: [t_{i-1},t_{i}] \rightarrow M$ is a constant speed local geodesic.
\item If $\alpha_i:U \rightarrow \mathcal{G}$ is a local section of $s$ at $h_i$, then for $\epsilon >0$ small enough, the concatenation of $c_{i}|_{[t_i-\epsilon,t_i]}$ and $r \circ \alpha_{i} \circ c_{i+1}|_{[t_{i}, t_{i}+\epsilon]}$ is a constant speed local geodesic.
\end{enumerate}
Any $\mathcal{G}$-path that is equivalent to a $\mathcal{G}$-geodesic is a $\mathcal{G}$-geodesic. Furthermore, if the induced length space structure on the space of orbits $W$ is complete, we can use the same type of arguments as in the previous subsection to prove that any $\mathcal{G}$-geodesic $c(t)$ is defined for all $t \in \mathbb{R}$ exist. We can define the notion of Jacobi fields along $\mathcal{G}$-geodesics and the notion of conjugate points in the obvious way. Also, one can check that the index lemma and Rauch's comparison theorem can be easily extended to this context. In particular, if the sectional curvature of the $\mathcal{G}$ invariant Riemannian metric $g$ is bounded above by a constant $K >0$, then there are no conjugate points along $\mathcal{G}$-geodesics with length less than equal to $\frac{\pi}{\sqrt{K}}$.

     Let $p$ be some point in $M$. If we assume as before that the induced metric space structure on the space $W$ is complete, then we can generalize the notion of the exponential map by defining $\mathrm{Exp}_{p}$ to be a collection of maps $\phi : U \rightarrow V$ where $U$ is an open subset of the tangent space $T_{p}M$, $V$ is an open subset of $M$ and $\psi$ is a map that corresponds to flowing along a $\mathcal{G}$-geodesic. More precisely, such a map $\psi$ can be expressed as a composition of exponential maps (in the regular sense) with elements $\psi$ of the pseudogroup $\mathbf{H}$ generated by $\mathcal{G}$. If $r$ is small enough so that there are no conjugate points on $\mathcal{G}$-geodesics emanating from $p$ and of length less than $r$, then the maps $\psi$ will be local diffeomorphisms if we restrict them to open subsets $U$ of the open ball $B_r \subset T_{p}M$ centered at the origin and of radius $r$. Hence, given any $\mathcal{G}$-invariant tensor, we can pull back this tensors to a tensor on $B_r$ by using the maps $\psi$. The pullback will not depend on the choice of $\psi$ since the tensor is $\mathcal{G}$-invariant. In particular, we can pullback the Riemannian metric $g$ to a Riemannian metric on $B_r$. The geodesics of this pullback Riemannian metric which emanate from the origin will clearly be straight lines.

	\section{Uniqueness of solutions to the Ricci flow.}
	We will now consider the Ricci flow on closed Riemannian groupoids. We first start by proving a uniqueness theorem.
	
\begin{thm} \label{uniqueness-theorem}
Suppose $g_{0}$ is a $\mathcal{G}$-invariant Riemannian metric on a closed groupoid $(\mathcal{G},M^{n})$ such that the induced metric $d_{0}$ on the space of orbits $W= M/\mathcal{G}$ is complete. Suppose that there exists a Haar system $\rho$ and a constant $C_{1}>0$ such that the measure $vol_{0}$ induced by $\rho$ and $g_{0}$ on the space $W$ satisfies the growth condition
\begin{align}
vol_{0}(B(\mathcal{O},r)) \leq C_{1} e^{C_{1}r}  \label{volumeGrowth}
\end{align}
 where $B(\mathcal{O},r)$ is a ball of radius $r > 0 $ in $(W,d_{0})$. Suppose also that the mean curvature form $\theta$  of the Haar system satifies
\begin{align}
 |\theta|_{g(0)} \leq C_{2} \label{thetaControlDos}
\end{align}
  for some constant $C_2 >0$. Finally, suppose that $g(t)$ and $\tilde{g}(t)$ are $\mathcal{G}$-invariant solutions to the Ricci flow on the closed interval $[0,T]$ with initial value $g(0)=\tilde{g}(0)=g_{0}$ and satisfying 
\begin{align}
\displaystyle \mathrm{sup}_{M \times [0,T]} |Rm|_{g(t)} , \mathrm{sup}_{M \times [0,T]} |\widetilde{Rm}|_{g(t)} \leq K \label{CurvatureControlDos}
\end{align}
for some constant $K \geq 0$. Then $g(t)=\tilde{g}(t)$ for all $t \in [0,T]$.
\end{thm}
The conditions (\ref{volumeGrowth}), (\ref{thetaControlDos}) and (\ref{CurvatureControlDos}) will hold for any choice of a Haar system $\rho$ if $W$ is compact.

\begin{crly}
Suppose $g_{0}$ is a $\mathcal{G}$-invariant Riemannian metric on a closed groupoid $(\mathcal{G},M^{n})$ with compact space of orbits $W$. Suppose that $g(t)$ and $\tilde{g}(t)$ are $\mathcal{G}$-invariant solutions to the Ricci flow on the closed interval $[0,T]$ with initial value $g(0)=\tilde{g}(0)=g_{0}$. Then $g(t)=\tilde{g}(t)$ for all $t \in [0,T]$.
\end{crly}

We use the energy approach of \cite{kotschwar} to prove the theorem. Kotschwar defines a time-dependent energy integral of the form
\begin{eqnarray*}
\mathcal{E}(t)=\int_{M} \left(t^{-1}|g-\tilde{g}|_{g(t)}^2+ t^{-\alpha}|\Gamma -\tilde{\Gamma}|_{g(t)}^{2} + |Rm-\widetilde{Rm}|_{g(t)}^2 \right) \, e^{-c} d\mu_{g(t)}.
\end{eqnarray*}
The function $c$ in the above equation depends on the distance function and satisfies certain properties and $\alpha$ is a suitable constant. We will use a similar function in our proof. It will be defined in terms of the distance function of the space $W$. Also $\Gamma$ and $\tilde{\Gamma}$ denote the Christoffel symbols and $Rm$  and $\widetilde{Rm}$ denote the curvature tensors of $g$ and $\tilde{g}$ respectively. Since we will work with $\mathcal{G}$-invariant geometric quantities, we will use the measure $\phi^2 d\mu_{g}$ where $\phi$ is a nonnegative cutoff function for the Haar system $\rho$. Before starting the proof of Theorem~\ref{uniqueness-theorem}, we collect the results of \cite{kotschwar} that we will need.
  
	We will use the metric $g(t)$ as a reference metric and will usually  use $|.|=|.|_{g(t)}$  to denote the norms induced on $T^{k}_{l}(M)$ by $g(t)$. Just as in \cite{kotschwar}, we set
	\begin{align*}
	h =g-\tilde{g}, \, A= \nabla -\tilde{\nabla}, \, S= Rm -\widetilde{Rm}.
	\end{align*}
More precisely, $A_{ij}^{k}=\Gamma_{ij}^{k}-\tilde{\Gamma}_{ij}^{k}$ and $S_{ijk}^{l}=R_{ijk}^{l}-\tilde{R}_{ijk}^{l}$. The following geometric inequalities are listed in \cite{kotschwar}[Subsection 1.1]:
\begin{eqnarray}
\left|\frac{\partial h}{\partial t}\right| \leq C|S| \label{kcontrol}
\end{eqnarray}

\begin{eqnarray}
\left|\frac{\partial  A}{\partial t}\right| \leq C \left(|\tilde{g}^{-1}||\tilde{\nabla}\widetilde{Rm}||h| + |\widetilde{Rm}||A| +|\nabla S| \right) \label{Acontrol}
\end{eqnarray}
and
\begin{eqnarray}
\left| \frac{\partial S}{\partial t}-\nabla S -\mathrm{div}U \right|&\leq& C \left(|\tilde{g}^{-1}||\tilde{\nabla}\widetilde{Rm}|A| + |\tilde{g}^{-1}||\widetilde{Rm}|^2|h|\right)  \label{Scontrol}   \\
																													&   &	+ C(|Rm|+ |\widetilde{Rm}|)|S|  \notag
\end{eqnarray}
	where, in each of these inequalties, $C$ is a constant depending only on $n$. The quantity $U$ is a (3,2)-tensor given by $U_{ijk}^{ab}=g^{ab}\nabla_{b}\tilde{R}_{ijk}^{l}-\tilde{g}^{ab}\tilde{\nabla}_{b} \tilde{R}_{ijk}^{l}$ and satisfies 
\begin{eqnarray}
		|U| \leq C \left(|\tilde{g}^{-1}||\tilde{\nabla}\widetilde{Rm}||h| + |A||\widetilde{Rm}| \right). \label{Ucontrol}
\end{eqnarray}
The uniform bound on the curvature (\ref{CurvatureControlDos}) imply that the metrics $g(t), \tilde{g}(t)$ and $g_{0}$ are uniformly equivalent. Hence, the uniform bound on the mean curvature form at time $t = 0$ implies that $|\theta| = |\theta|_{g(t)} \leq N$ for some positive constant $N$ depending on $C_2,K$ and $T$.
					
 In the case of a complete Riemannian metric $g_{0}$, we can use the curvature bound (\ref{CurvatureControlDos}) and the estimates of Bando and Shi (see \cite{bando,shi} and \cite{chow2}[Chapter 14]) to obtain an estimate
	\begin{eqnarray}
	|Rm|_{g(t)} + |\tilde{Rm}|_{\tilde{g}(t)} + \sqrt{t}|\nabla Rm|_{g(t)} + \sqrt{t}|\tilde{\nabla}\widetilde{Rm}|_{\tilde{g}(t)}  \leq N\label{l'estimation}
	\end{eqnarray}
	for all $ t \in  [0,T]$ and for some constant $N=N(n,K,T^{\ast})$ with $T^{\ast}=\max \{T,1\}$. We still have such an estimate under the simpler assumption that $(W,d_{0})$ is complete. Indeed, if we fix $r > 0$ with $r \leq \frac{\pi}{\sqrt{K}}$, then, as we mentioned in the last subsection, $\mathcal{G}$-geodesics with respect to the metric $g(0)$ of length less than or equal to $r$ will not contain conjugate points. Therefore, for any $p \in M$, we can use the exponential $\mathrm{Ex}_{p}$ as defined in the previous subsection to pullback the flows $g(t)$ and $\tilde{g}(t)$ to the open ball $B_r$ in $T_{p}M$ centered at the origin and of radius $r$. Furthermore, for the pullback of the time t=0 metric and for any $0< s <r$, the corresponding closed ball centered at the origin and of radius $s$ is just the closed ball of radius $s$ for the inner product space $(T_{p}M, g_p(0))$ and is therefore compact. We can then apply the Bando and Shi estimates to the pullback flow on $B_r$. This then gives us an estimate of the form (\ref{l'estimation}) on $M$. Finally, by using the fact that the metrics $g(t), \tilde{g}(t)$ and $g_0$ are uniformly equivalent, we see that, for some possibly larger constant $N$, the inequality (\ref{l'estimation}) remains valid when the norms are replaced by the norm $|.|=|.|_{g(t)}$. It follows that \cite{kotschwar}[Lemma 6]  holds. More precisely, we have  

	\begin{eqnarray}
	|h(p,t)| \leq Nt  \,\, \text{and}  \,\, |A(p,t)| \leq N \sqrt{t} \label{more-control}
	\end{eqnarray}
	on $M \times [0,T]$  for some constant $N=N(n,K,T^{\ast})$.
	
	We now define the decay function $c$. We fix an orbit $\mathcal{O}$ in $W$ and we define a function $\overline{r}$ on $M$ by $\overline{r}(x) = d_{0}(\mathcal{O}, \mathcal{O}_{x})$. This function is Lipschitz and hence it is differentiable almost everywhere on $M$. We have $|\nabla \overline{r}|^2_{g(0)} \leq 1$. Hence $|\nabla \overline{r}|^2_{g(t)} \leq N$ for some constant $N$ since the metrics $g(t)$ and $g(0)$ are uniformly equivalent. For some constant $B >0$, we define the decay function by $c(x,t) =  \frac{B\overline{r}^2(x)}{(4(2T-t))}$ for $(x,t) \in M \times [0,T]$. It follows that $c$ is differentiable almost everywhere and we can choose $B$ small enough so that 
\begin{align}
\frac{\partial c}{\partial t} - 5|\nabla c|^2_{g(t)} \geq 0   \label{decayFunctioncontrol}
\end{align}
	We also have the inequality 
\begin{align}
	e^{-c(x,t)} \leq e^{\frac{-B\overline{r}^2}{8T}}. \label{DecayControl}
\end{align}
We will denote by $vol_{t}$ the measure on $W$ induced by the Haar system $\rho$ and the Riemannian metric $g(t)$. So, for any open ball $B(\mathcal{O},r)$ in $(W, d_{0})$, we have 
\begin{eqnarray*}
\mathrm{vol}_{t}(B(\mathcal{O},r)) = \int_{\sigma^{-1}(B(\mathcal{O},r))} \, \phi^2 d\mu_{g}
\end{eqnarray*}
 where $\sigma: M  \rightarrow W$ is the quotient map. Since $\frac{d\mu_{g}}{dt} = -R d\mu_{g}$ where $R$ is the scalar curvature and since we have a uniform bound on the curvature, we see that the volume growth condition (\ref{volumeGrowth}) extends to a volume growth condition
\begin{align}
vol_{t}(B(\mathcal{O},r)) \leq Ne^{Nr}  \label{volumeGrowth2}
\end{align}
for some constant $N$ depending on $C_1$,$K$ and $T$ and the dimension $n$. This volume growth condition and the estimates (\ref{kcontrol}) to (\ref{DecayControl}) imply that, for some fixed $\alpha \in (0,1)$, the energy integral
	\begin{eqnarray}
	\mathcal{E}(t)=\int_{M} \left(t^{-1}|h|^2+ t^{-\alpha}|A|^{2} + |S|^2 \right) \, e^{-c} \phi^2 d\mu_{g}
	\end{eqnarray}
	is well-defined and differentiable on $(0,T]$. Furthermore, this energy integral satisfies $\displaystyle \lim_{t \rightarrow 0^{+}} \mathcal{E}(t)=0$. Theorem~\ref{uniqueness-theorem} now follows from iterating the following result.
	
	\begin{ppr} [\cite{kotschwar} Proposition 7]
	There exists $N=N(n,K,T^{\ast}) >0$ and $T_{0}=T_{0}(n,\alpha) \in (0,T]$ such that $\mathcal{E}^{\prime}(t) \leq  N \mathcal{E}(t)$ for all $t \in (0,T_{0}]$. Hence $\mathcal{E} \equiv 0$ on $(0,T_{0}]$.
	\end{ppr}
	\begin{proof}
	Since the proof is basically the same as the proof of \cite{kotschwar}[Proposition 7], we will, for the most part, not be repeating the computations that are similar. We will mostly use the same notation and convention. As before, let $R=R(g(t))$ be the scalar curvature of the metric $g(t)$. We define
	\begin{eqnarray*}
	\mathcal{S}(t)=\int_{M} |S|^2 \,  e^{-c}\phi^2 d\mu_{g}, \, \, \mathcal{H}(t)= t^{-1}\int_{M} |h|^2 \, e^{-c} \phi^2 d\mu_{g},\\
	\mathcal{I}(t)= t^{-\alpha} \int_{M} |A|^2 \,  e^{-c} \phi^2 d\mu_{g}, \,\, \text{and} \,\, \mathcal{J}(t)=\int_{M} |\nabla S|^2 \,  e^{-c} \phi^2 d\mu_{g}
	\end{eqnarray*}
	so $\mathcal{E}(t)=\mathcal{S}(t)+\mathcal{H}(t)+\mathcal{I}(t)$. We will denote by $C$ a series of constants depending only on $n$ and by $N$ a series of constants depending on at most $n,\beta, K$ and $T^{\ast}$.
	If we apply the condition (\ref{CurvatureControlDos}), we obtain
	\begin{eqnarray*} 
	\mathcal{S}^{\prime}  &\leq& -\int_{M} |S|^2 R \, e^{-c}\phi^2 d\mu_{g} +  \int_{M} \left( 2\langle \frac{\partial S}{\partial t},S \rangle -\frac{\partial c}{\partial t} |S|^2 \right) \, e^{-c} \phi^2 d\mu_{g}\\
	                      & \leq&  N \mathcal{S} + \int_{M} \left( 2\langle \frac{\partial S}{\partial t},S \rangle -\frac{\partial c}{\partial t} |S|^2 \right) \, e^{-c}\phi^2 d\mu_{g}.
	\end{eqnarray*}
	 Just like in the proof of \cite{kotschwar}[Proposition 7], we can use (\ref{Scontrol}) and (\ref{l'estimation}) to reduce this to 
	\begin{eqnarray*} 
	\mathcal{S}^{\prime} & \leq& N \mathcal{S} + t\mathcal{H} + t^{\alpha -1}\mathcal{I} + \int_{M} \left(2\langle \Delta S + \mathrm{div}U,S \rangle -\frac{\partial c}{\partial t} |S|^2 \right) \, e^{-c}\phi^2 d\mu_{g}.
	\end{eqnarray*}
	 We now use the integration by parts formula (\ref{integpart}) to rewrite the integral
	\begin{eqnarray*}
	2\int_{M}\langle \Delta S + \mathrm{div}U,S \rangle \,  e^{-c}\phi^2 d\mu_{g}
	\end{eqnarray*}
	as
	\begin{eqnarray*}
	-2 \int_{M} |\nabla S|^2 \, e^{-c}\phi^2 d\mu_{g} + 2 \int_{M} \langle \nabla S, \nabla c \otimes S \rangle \, e^{-c}\phi^2 d\mu_{g} \\
	 + 2 \int_{M} \langle \iota_{\theta^{\sharp}}\nabla S, S \rangle \, \phi^2d\mu_{g} 
	-2 \int_{M} \langle U, \nabla S \rangle \, e^{-c} \phi^2d\mu_{g}  \\
	+ 2 \int_{M} \langle U, \nabla c \otimes S \rangle \, e^{-c} \phi^2d\mu_{g}
	+ 2 \int_{M} \langle \iota_{\theta^{\sharp}}U, S \rangle \, e^{-c} \phi^2d\mu_{g} 
\end{eqnarray*}
	If we use the uniform bound on the mean curvature form, we get
\begin{eqnarray*}
	2\langle \iota_{\theta^{\sharp}}\nabla S, S \rangle&=& 2\langle \nabla S, \theta \otimes S \rangle \\
	                                                   &\leq & 2 |\nabla S||\theta||S| \\
																										  &\leq & \frac{|\nabla S|^2}{4} + 4 |\theta|^2|S|^2\\
																											 &\leq & \frac{|\nabla S|^2}{4}  + N|S|^2	
	\end{eqnarray*}
	
	\begin{eqnarray*}
	2\langle \iota_{\theta^{\sharp}}U, S \rangle \leq  |U||S||\theta| \leq |U|^2 + N|S|^2 
	\end{eqnarray*}
	
	\begin{eqnarray*}
	2 \langle \nabla S, \nabla c \otimes S \rangle \leq \frac{|\nabla S|^2}{4} + 4 |\nabla c|^2 |S|^2
	\end{eqnarray*}
	
	\begin{eqnarray*}
	2 \langle U, \nabla c \otimes S \rangle \leq |U|^2 + |\nabla c|^2 |S|^2
	\end{eqnarray*}
	
	and
	\begin{eqnarray*}
	 -2\langle U, \nabla S \rangle \leq \frac{|\nabla S|^2}{2}  + 2|U|^2.	
	\end{eqnarray*}
	Therefore,
		\begin{eqnarray*}
	2\int_{M}\langle \Delta S + \mathrm{div}U,S \rangle \,  e^{-c} \phi^2 d\mu_{g} & \leq & -\mathcal{J} + N\mathcal{S} \\
	                                     & + & \int_{M} \left(4|U|^2 + 5|\nabla c|^2 |S|^2 \right) \, e^{-c} \phi^2 d\mu_{g}
	\end{eqnarray*}
	It follows from (\ref{Ucontrol}) and (\ref{l'estimation}) that we have $|U|^2 \leq Nt^{-1}|h|^2 + N|A|^2$. Hence, if we combine this with the condition $\frac{\partial c}{\partial t} \geq 5|\nabla c|^2_{g(t)}$, we obtain 
	\begin{eqnarray}
	\mathcal{S}^{\prime} \leq N \mathcal{S} + (t+N)\mathcal{H}  + (t^{\alpha-1}+Nt^{\alpha})\mathcal{I} -\mathcal{J} \notag \\
	                     \leq  N \mathcal{S} + N\mathcal{H}  + (t^{\alpha-1}+N)\mathcal{I} -\mathcal{J}  \label{Sprime}
	\end{eqnarray}
where the last line was obtained by using the inequalities $t \leq T$ and $t^{\alpha} \leq T^{\alpha}$.

       For the quantities $\mathcal{H}$ and $\mathcal{I}$, we apply the uniform bounds on the curvature and the mean curvature form as well as the condition $\frac{\partial c}{\partial t} \geq 0$ to obtain:
\begin{eqnarray*}
\mathcal{H}^{\prime} \leq (N- t^{-1})\mathcal{H} + 2t^{-1} \int_{M} \langle \frac{\partial h}{\partial t}, h \rangle  \, \phi^2 d\mu_{g}
\end{eqnarray*}
and 
\begin{eqnarray*}
\mathcal{I}^{\prime} \leq (N- \beta t^{-1})\mathcal{I} + 2t^{-\alpha} \int_{M} \langle \frac{\partial A}{\partial t}, A \rangle  \, \phi^2 d\mu_{g}.
\end{eqnarray*}
	We then use the conditions (\ref{kcontrol}), (\ref{Acontrol}) and (\ref{l'estimation}) and do the same computations as in \cite{kotschwar}[Proposition 7] to obtain
	\begin{eqnarray}
	\mathcal{H}^{\prime} \leq \left(N-\frac{t^{-1}}{2}\right)\mathcal{H} + C\mathcal{S} \label{Kprime}
	\end{eqnarray}
	and 
	\begin{eqnarray}
	\mathcal{I}^{\prime} \leq N \mathcal{H} + (N-\alpha t^{-1}+Ct^{-\alpha})\mathcal{I} + \mathcal{J} \label{Iprime}.
	\end{eqnarray}
	Combining (\ref{Sprime}),(\ref{Kprime}) and (\ref{Iprime}), we get
	\begin{eqnarray*}
	\mathcal{E}^{\prime}(t) \leq N \mathcal{E}(t) -\frac{1}{2}t^{-1}\mathcal{H}(t)-t^{-1}(\alpha -t^{\alpha} + Ct^{1-\alpha})\mathcal{I}(t)
	\end{eqnarray*}
	Therefore, for $T_0$ sufficiently small depending only on $\alpha$ and $C=C(n)$, and for some large enough $N=N(n,K,\alpha,T)$, we have $\mathcal{E}^{\prime}(t) \leq N \mathcal{E}(t)$ on $(0,T_{0}]$. Since $ \displaystyle \lim_{t \rightarrow 0^{+}} \mathcal{E}(t)=0$, it follows from Gronwall's inequality that $\mathcal{E} \equiv 0$ on $(0,T_{0}]$. This completes the proof.
	\end{proof}

	\section{Short time existence of solutions to the Ricci flow}
	We will use the Ricci-deTurck flow to prove short time existence of the Ricci flow on a closed groupoid.
	
	\begin{thm} \label{existence}
	If $(\mathcal{G},M)$ is a closed groupoid with compact connected space of orbits $W=M/\mathcal{G}$ and if $g_{0}$ is a $\mathcal{G}$-invariant Riemannian metric on $M$, then, for some $T>0$, there exists a Ricci flow solution $g(t)$ defined for $t \in [0,T)$ and with initial condition $g(t)=g(0)$.
	\end{thm}
	
	\begin{proof}
	 Let $\tilde{g}$ be a fixed $\mathcal{G}$-invariant background metric. Recall that the Ricci-deTurck flow is defined by
	\begin{eqnarray}
		\frac{\partial g_{ij}}{\partial t} &=& -2 R_{ij} + \nabla_{i} B_{j} + \nabla_{j}B_{i} \label{DeTurck1} \\
	g(0)&=& g_{0} \notag
	\end{eqnarray}
	where $B=B(g)$ is the time-dependent one-form defined by 
	\begin{eqnarray*}
	B_{j}=g_{jk}g^{pq}(\Gamma^{k}_{pq}-\tilde{\Gamma}^{k}_{pq})
	\end{eqnarray*}
	This is a strictly parabolic partial differential equation. In fact, as shown in \cite{chow1}[Chaper 7], we can rewrite this equation as 
	\begin{eqnarray}
	\frac{\partial g_{ij}}{\partial t}= P(g)= g^{pq}\tilde{\nabla}_{p}\tilde{\nabla}_{q}g_{ij} + K(g,\tilde{\nabla}g)  \label{DeTurck2}
	\end{eqnarray}
where 
	\begin{eqnarray*}
    K(g,\tilde{\nabla}g) &=& -g^{pq}g_{ik}\tilde{g}^{kl}\tilde{R}_{jplq}-g^{pq}g_{jk}\tilde{g}^{kl}\tilde{R}_{iplq}\\
		&& +g^{pq}g^{kl}(\frac{1}{2}\tilde{\nabla}_{i}g_{kp}\tilde{\nabla}_{j}g_{lq} + \tilde{\nabla}_{p}g_{jk}\tilde{\nabla}_{l}g_{iq}-\tilde{\nabla}_{p}g_{jk}\tilde{\nabla}_{q}g_{il})  \\    &	&			-g^{pq}g^{kl}(\tilde{\nabla}_{j}g_{kp}\tilde{\nabla}_{q}g_{il} + \tilde{\nabla}_{i}g_{kp}\tilde{\nabla}_{q}g_{jl}). 
	\end{eqnarray*}
	
	Since the manifold $M$ is not necessarily compact, we can't  directly use the standard theorems for parabolic partial differential equations. However, since the equation~(\ref{DeTurck2}) consists of $\mathcal{G}$-invariant quantities, we can relate this equation to a partial differential equation defined for sections of a vector bundle with compact base space.
	
	 We showed in Subsection~\ref{subsec:induced-groupoid} that the orthonormal frame bundle $F=F_{\tilde{g}}$ over $M$ defined by the metric $\tilde{g}$ and its quotient $Z= F/\hat{ \mathcal{G}}$ can be equipped with metrics $\hat{g}$ and $\overline{g}$ respectively such that the maps $\pi: (F,\hat{g}) \rightarrow (M,g)$ and $\hat{\sigma}: (F,\hat{g}) \rightarrow (Z,\overline{g})$ are Riemannian submersions. Given any $\mathcal{G}$-invariant vector bundle $E$ over $M$ equipped with a $\mathcal{G}$-invariant connection $\tilde{\nabla}^{E}$, we can consider the pull back bundle $\hat{E}=\pi^{\ast}E$ of $E$ over $F$ equipped with the pull back connection $\hat{\nabla}=\hat{\nabla}^{E}=\pi^{\ast} \tilde{\nabla}^{E}$ which will be $\mathcal{G}$-invariant. Also, the $O(n)$ action on $F$ induces an action on $\hat{E}$ which also preserves the connection $\hat{\nabla}^{E}$.
	
	  We apply this to the bundle $E=Sym(T^{\ast}M)$ of symmetric (2,0)-tensors on $M$ where we equip it with the connection induced by the background $\mathcal{G}$-invariant metric $\tilde{g}$ on $M$. We now define a parabolic partial differential equation for sections of the corresponding bundle $\hat{E}$ as follows.

		 First, we denote by $\mathcal{V}$ and by $\mathcal{H}$ the vertical and horizontal distributions for the Riemannian submersion $\pi: (F,\hat{g}) \rightarrow (M,g)$. Let $(x_1,\cdots,x_{n})$ be local coordinates defined on an open set $U$ of $M$. For $ 1 \leq i \leq n$, we can lift the vector field $\partial_{ x_i}$ to a horizontal vector field $w_i$ on $F$. Hence, $\{w_1, \cdots, w_{n}\}$ is a local frame for $\mathcal{H}|_{\pi^{-1}(U)}$. Let $\{v_1,\cdots, v_{d}\}$ be a local frame for $\mathcal{V}|_{\pi^{-1}(U)}$ where $d= \mathrm{dim}O(n)=\frac{n(n-1)}{2}$.  The bi-invariant metric along the $O(n)$-fiber is then given by $\hat{g}_{\mu \nu}=\hat{g}(v_{\mu},v_{\nu})$.
				If $s$ is a section of $\hat{E}$, then for each $f \in F$, $s(f)=s_{f}$ is an element of $Sym(T^{\ast}_{x}M)$ where $x =\pi(f)$. If $x \in U$, then $s_{f}$ is locally given by $s_{f;ij}=s_{f}(\partial_{x_i},\partial_{ x_j})$. We will suppress the index $f$ and simply write $s_{ij}$. If $s_{ij}$ is positive definite so that the inverse matrix $s^{ij}$ is defined, the partial differential equation is then locally given by
		\begin{eqnarray}
		\frac{\partial s_{ij}}{\partial t} = \hat{P}(s)= s^{pq}\hat{\nabla}_{p}\hat{\nabla}_{q}s + \hat{g}^{\mu \nu}\hat{\nabla}_{\mu}\hat{\nabla}_{\nu}s + \hat{K}(s,\hat{\nabla}s) \label{DeTurck3}
		\end{eqnarray}
		where
		\begin{eqnarray*}
		 \hat{K}(s,\hat{\nabla}s) &=& -s^{pq}s_{ik}\tilde{g}^{kl}\tilde{R}_{jplq}-s^{pq}s_{jk}\tilde{g}^{kl}\tilde{R}_{iplq}\\
		&& +s^{pq}s^{kl}(\frac{1}{2}\hat{\nabla}_{i}s_{kp}\hat{\nabla}_{j}s_{lq} + \hat{\nabla}_{p}s_{jk}\hat{\nabla}_{l}s_{iq}-\hat{\nabla}_{p}s_{jk}\hat{\nabla}_{q}s_{il})  \\   			
		& &  -s^{pq}s^{kl}(\hat{\nabla}_{j}s_{kp}\hat{\nabla}_{q}s_{il} + \hat{\nabla}_{i}s_{kp}\hat{\nabla}_{q}s_{jl}). 
	\end{eqnarray*}
	
		Here, the derivative terms  with greek letters as indices such as $\hat{\nabla}_ {\mu}$ denote differentiation along the vertical vectors $v_{\mu}$, whereas the derivative terms with latin letters as indices denote differentiaton along the horizontal vectors $w_{i}$. Equation~(\ref{DeTurck3}) is a strictly parabolic partial differential equation and is  invariant under the actions of $O(n)$ and $\hat{\mathcal{G}}$. In fact, the elliptic operator $\hat{P}$  is, in a sense, simply the operator $P$  plus the Laplacian along the $O(n)$-fiber. Furthermore, if $s$ is an $O(n)$-invariant section, meaning if $s$ is the pull back of a section $g$ of $E$,  we will have $\hat{P}(s)=\pi^{\ast}(P(g))$. Hence, finding a $\mathcal{G}$-invariant solution of (\ref{DeTurck2}) is the same as finding an $O(n)$ and $\hat{\mathcal{G}}$-invariant solution of (\ref{DeTurck3}).

	Note that the vector bundle $\hat{E}$ over $F$ induces a vector bundle $\overline{E}$ over $Z=F/\hat{\mathcal{G}}$ which is equipped with a connection $\overline{\nabla}=\overline{\nabla}^{E}$ induced by the connection $\hat{\nabla}^{E}$.  The action of $O(n)$ on $\hat{E}$ induces an action of $O(n)$ on $\overline{E}$ and the connection $\overline{\nabla}$ is invariant under this action.  Furthermore, the parabolic partial differential equation (\ref{DeTurck3}) can be related to a parabolic partial differential equation 
	\begin{eqnarray}	
\frac{\partial u}{\partial t} &=& \overline{P}(u)  \label{DeTurck4}
\end{eqnarray}
	where $\overline{P}$ is an $O(n)$-invariant elliptic partial differential operator acting on sections of the bundle $\overline{E}$.  So, if we set the initial condition to be     the $O(n)$-invariant section $u_0$ of $\overline{E}$ corresponding to $g_{0}$ and since $Z$ is a closed manifold, we will have an $O(n)$-invariant solution $u(t)$ of (\ref{DeTurck4}) defined on $[0,T)$ for some $T >0$. This will give us a solution $g(t)$ of (\ref{DeTurck1}).
	
	Finally, if we consider again the time dependent $\mathcal{G}$-invariant one-form $B$ that was used in the Ricci-DeTurck flow, the dual vector field $B^{\sharp}$ generates a differentiable equivalence $\Psi_{t}$ for short enough time $t$  as described in Subsection~\ref{subsec:Gvectorfields}. Furthermore, the formula
\begin{eqnarray*}
	\frac{d}{dt}\arrowvert_{t=t_0} \Psi_{t}^{\ast}\omega_t= \Psi_{t_0}^{\ast}\left(\mathcal{L}_{X}\omega + \frac{d}{dt}\arrowvert_{t=t_0}\omega_{t} \right)
\end{eqnarray*}
will hold for any time-dependent $\mathcal{G}$-invariant tensor $\omega_t$. Applying this to the solution $g(t)$ of (\ref{DeTurck1}), we see that $g^{\prime}(t)=\Psi_{t}^{\ast}g(t)$ is a solution of the Ricci flow. This completes the proof.
\end{proof}

	Just as in the closed manifold case, one can show that if $(\mathcal{G},M)$ is a closed groupoid with compact connected space of orbits $W=M/\mathcal{G}$ and if $g(t)$ is a solution to the Ricci flow on a maximal time interval $[0,T)$ with $T < \infty$, then 
	\begin{eqnarray*}
	sup_{M} |Rm|(.,t) \rightarrow \infty
	\end{eqnarray*}
	as $t \uparrow T$.
		The proof is the same. Indeed, we start by assuming that we have a uniform bound on the curvature and show that this implies that the Ricci flow solution can be extended past the time $T$, thus contradicting the maximality of $[0,T)$. The uniform bound on the curvature implies that the Riemannian metrics $g(t)$ are uniformly equivalent. More precisely, for some $K>0$, we will have
	\begin{eqnarray*}
	e^{-2Kt}g(0) \leq g(t) \leq e^{2Kt}g(0)
	\end{eqnarray*}
	for all $t \in [0,T]$. This implies that $g(t)$ can be extended continuously to the time interval $[0,T]$. As we mentioned in Section 3, the Bernstein-Bando-Shi global curvature estimates will hold in the case of a closed groupoid with compact connected space of orbits. Then, just as in the closed manifold case, we use these estimates to show that the extension of $g(t)$ to the closed interval $[0,T]$ is smooth. We can now take $g(T)$ as the initial metric in the short-time existence theorem in order to extend the flow to a Ricci flow on $[0,T+\epsilon)$ for some $\epsilon >0$ giving us the contradiction.

\section{The $\mathcal{F}$-functional and Ricci solitons on groupoids}
Let $(\mathcal{G},M)$ be a closed groupoid with compact connected orbit space $W=M/\mathcal{G}$. As we mentioned in the Introduction, $\mathcal{M}_{\mathcal{G}}$ will denote the space of $\mathcal{G}$-invariant Riemannian metrics on $M$ and we will simply denote by $\rho$ an element of the family of Haar systems $\Xi$ defined in Subsection~\ref{subsec:evolution-formulas}.

We consider the $\mathcal{F}$-functional $\mathcal{F}:  \mathcal{M}_{\mathcal{G}} \times \Xi  \rightarrow \mathbb{R}$ by
\begin{eqnarray}
\mathcal{F}(g, \rho) =\int_{M} \left(R + |\theta |^2 \right)\,  \phi^2d\mu_{g}  \label{Ffunctional}
\end{eqnarray}
where $R$ is the scalar curvature of the metric $g$, $d\mu_{g}$ the Riemannian density, $\theta$ the mean curvature form for $\rho$ and $\phi$ a corresponding nonnegative cutoff function.                         

  If we fix a representative $\rho_{0}$ of $\Xi$ and induce a bijection $C^{\infty}_{\mathcal{G}}(M) \rightarrow \Xi$ as described at the end of Subsection~\ref{subsec:evolution-formulas}, this can be viewed as a functional on $\mathcal{M}_{\mathcal{G}} \times C^{\infty}_{\mathcal{G}}(M)$
\begin{eqnarray}
\mathcal{F}(g, f) =\int_{M} \left(R + |\nabla f + \theta_0 |^2 \right)e^{-f} \,  \phi_0^2 d\mu_{g}  \label{Ffunctional2} 
 \end{eqnarray}
where $\theta_0$ is the mean curvature class for $\rho_{0}$ and $\phi_{0}$ a compatible cutoff function. Notice that, in the case of Example~\ref{groupExample}, this reduces to the usual $\mathcal{F}$-functional on a closed Riemannian manifold.

Before computing the variation of $\mathcal{F}$, we define, for convenience, a modified divergence operator by
\begin{eqnarray*}
\underline{\mathrm{div}}{\alpha}= \mathrm{div}\alpha -\iota_{\theta^{\sharp}}\alpha
\end{eqnarray*}
for any tensor $\alpha$ on $M$. Then the integration by parts formula~(\ref{integpart}) can now be written as 
\begin{eqnarray}
\int_{M} \langle \underline{\mathrm{div}}(\alpha), \omega \rangle \, \phi^2 d\mu_{g} &=& -\int_{M} \langle \alpha, \nabla \omega\rangle \, \phi^2 d\mu_{g}  \label{integpartdos}
\end{eqnarray} 
for any $\mathcal{G}$-invariant tensors $\alpha$ and $\omega$ of type $(r,s)$ and $(r-1,s)$ respectively. We also define the modified Laplacian operator as $\underline{\Delta}=\underline{\mathrm{div}} \circ \nabla$ and we get the formula
\begin{eqnarray}
\int_{M} \langle \underline{\Delta}(\alpha), \chi \rangle \, \phi^2 d\mu_{g} &=& -\int_{M} \langle \nabla \alpha, \nabla \chi\rangle \, \phi^2 d\mu_{g}  \label{integparttres}
\end{eqnarray}
for any $\mathcal{G}$-invariant tensors $\alpha$ and $\chi$ of type $(r,s)$.
  
	Now suppose that we have the variations $\delta g =v$ and $\delta (d\rho) = \beta d\rho $ where $v=v_{ij}$ is a $\mathcal{G}$-invariant symmetric $(2,0)$-tensor and where $\beta$ is a $\mathcal{G}$-invariant function on $M$.

	\begin{ppr}
	The variation of the $\mathcal{F}$-functional is given by
		\begin{eqnarray*}
	\delta \mathcal{F}&=&-\frac{1}{2}\int_{M} \langle 2Ric + 2 \nabla \theta , v \rangle \,\phi^2d\mu_{g} \\
	&&+  \int_{M} (2 \underline{\mathrm{div}}(\theta) + |\theta |^2 + R)\left(\frac{V}{2}-\beta\right) \,\phi^2d\mu_{g}.
	\end{eqnarray*}
	or equivalently
	\begin{eqnarray*}
	\delta \mathcal{F}&=&-\frac{1}{2}\int_{M} \langle 2Ric + \mathcal{L}_{(\theta)^{\sharp}}g , v \rangle \,\phi^2d\mu_{g} \\
	&&+  \int_{M} (2\mathrm{div}(\theta)-|\theta|^2 + R)\left(\frac{V}{2}-\beta\right)\,\phi^2d\mu_{g}.
	\end{eqnarray*}
	\end{ppr}
	\begin{proof}
	Note that since $\theta$ is a closed 1-form, the tensor $\nabla \theta $ is a symmetric (2,0)-tensor. So the expression above makes sense.
	
	The proof of the proposition is a straightforward computation. If we apply the formula~(\ref{time-derivative}), we get 
	\begin{eqnarray*}
	\delta \mathcal{F}&=& \int_{M} \delta \left(R + |\theta |^2 \right) \,\phi^2d\mu_{g} + \int_{M} \left(R + |\theta |^2 \right)\left(\frac{V}{2}-\beta \right)\phi^2d\mu_{g}.
	\end{eqnarray*}
It is known that
\begin{eqnarray*}
\delta R= -\langle Ric,v \rangle - \Delta V + \mathrm{div}(\mathrm{div}(v)).
\end{eqnarray*}
We will rewrite this in terms of the modified divergence operator and the modified Laplacian operator. A simple computation gives us
\begin{eqnarray*}
\Delta V &=& \underline{\Delta}V  + \langle \theta, \nabla V \rangle \\
\mathrm{div}(\mathrm{div}(v))&=& \underline{\mathrm{div}}(\underline{\mathrm{div}}(v)) + 2 \langle \theta, \underline{\mathrm{div}}(v) \rangle + \langle \nabla \theta, v \rangle + v(\theta,\theta).
\end{eqnarray*}
So we get
\begin{eqnarray}
\delta R &=& -\langle Ric,v \rangle -\underline{\Delta}V  - \langle \theta, \nabla V \rangle + \underline{\mathrm{div}}(\underline{\mathrm{div}}(v))   \label{scalarvariation} \\
        & & + 2 \langle \theta, \underline{\mathrm{div}}(v) \rangle  + \langle \nabla \theta, v \rangle + v(\theta,\theta). \notag
\end{eqnarray}
On the other hand, the variation of $|\theta |^2$ is given by 
\begin{eqnarray*}
\delta(|\theta |^2)&=& \delta \left( g^{ij}\theta_{i}\theta_{j} \right) \\
                   &=&  -v^{ij}\theta_{i}\theta_{j}+2  g^{ij}\theta_{i}\nabla_{j} \beta \\
                   &=&  -v(\theta, \theta) +2 \langle \nabla \beta, \theta \rangle.
\end{eqnarray*}
Combining this with (\ref{scalarvariation}), we can rewrite the variation of $\mathcal{F}$ as 
\begin{eqnarray*}
\int_{M}\left(-\langle Ric, v\rangle -\underline{\Delta}V-\langle \theta , \nabla V \rangle + \underline{\mathrm{div}}(\underline{\mathrm{div}}(v)) \right)\phi^2d\mu_{g} \\
+ \int_{M}\left( 2 \langle \theta, \underline{\mathrm{div}}(v) \rangle  + \langle \nabla \theta, v \rangle + 2 \langle \nabla \beta, \theta \rangle \right) \phi^2d\mu_{g} \\
+\int_{M}(R + |\theta |^2 )\left(\frac{V}{2}-\beta \right)\phi^2d\mu_{g}. 
 \end{eqnarray*}									
Using the formulas (\ref{integpartdos}) and (\ref{integparttres}),we get
	
	\begin{eqnarray*}
	\int_{M} \underline{\Delta}V  \,\phi^2d\mu_{g}= \int_{M} \underline{\mathrm{div}}(\underline{\mathrm{div}}(v))\,\phi^2d\mu_{g} = 0 
	\end{eqnarray*}
	
	and 
	\begin{eqnarray*}
	-\int_{M} \langle \theta , \nabla V \rangle \,\phi^2d\mu_{g}&=& \int_{M} V  \underline{\mathrm{div}}(\theta)\phi^2d\mu_{g} \\
	2\int_{M} \langle \theta, \underline{\mathrm{div}}(v) \rangle \,\phi^2d\mu_{g} &=& -2\int_{M}  \langle \nabla \theta , v \rangle \, \phi^2d\mu_{g} \\
	2\int_{M} \langle \nabla \beta, \theta \rangle \,\phi^2d\mu_{g} &=& -2 \int_{M} \beta \underline{\mathrm{div}}(\theta) \,\phi^2d\mu_{g}. \\
	\end{eqnarray*}
																						
	So we end up with
\begin{eqnarray*}
\delta \mathcal{F}&=&	-\int_{M} \langle Ric + \nabla \theta, v \rangle \,\phi^2d\mu_{g}  \\
&&+ \int_{M} (2 \underline{\mathrm{div}}(\theta) + |\theta |^2 +R)\left(\frac{V}{2}-\beta\right)\,\phi^2d\mu_{g}.
\end{eqnarray*}	

This completes the proof.	
	\end{proof}	
	We now impose the condition
	\begin{eqnarray}
	\frac{V}{2}-\beta=0. \label{impose}
	\end{eqnarray}
	The variaton of the functional $\mathcal{F}$ is now
	\begin{eqnarray}
	\delta \mathcal{F}=-\frac{1}{2}\int_{M} \langle 2Ric + \mathcal{L}_{\theta^{\sharp}}g , v \rangle   \,\phi^2d\mu_{g}.
	\end{eqnarray}
	If we define a Riemannian metric on $\mathcal{M}_{\mathcal{G}}$ by 
\begin{eqnarray}
\langle v_{ij}, v_{ij} \rangle_{g} = \frac{1}{2} \int_{M} v^{ij}v_{ij} \phi^2 d\mu_{g}
\end{eqnarray}
the gradient flow of $\mathcal{F}$ on $\mathcal{M}_{\mathcal{G}}$ is given by
\begin{eqnarray}
\frac{\partial g}{\partial t} = -2(Ric(g) + \nabla \theta ) =-2Ric-\mathcal{L}_{\theta^{\sharp}}g						
 \end{eqnarray}

 and we will have
\begin{eqnarray}
\frac{d}{dt}\mathcal{F}=2 \int_{M} |Ric(g) + \frac{1}{2}\mathcal{L}_{\theta^{\sharp}}g|^2 \phi^2 d\mu_{g} \geq 0.
\end{eqnarray}
This is zero when 
\begin{eqnarray}
Ric(g) + \frac{1}{2}\mathcal{L}_{ \theta^{\sharp}}g=0.
\end{eqnarray}
This is the equation of a steady soliton. We will call such a soliton a \emph{groupoid steady soliton}. If $\theta$ is exact, this reduces to a gradient steady soliton.
  
	Let $\Psi_{t}$ be the time dependent self-equivalences of $(\mathcal{G},M)$ generated by the time dependent $\mathcal{G}$-invariant vector field $\theta^{\sharp}$.  We set $g^{\prime}(t)=\Psi_{t}^{\ast}g(t)$ and $\theta^{\prime}(t) = \Psi_{t}^{\ast}\theta(t)$. We will then have 
	\begin{eqnarray}
	\frac{\partial g^{\prime}}{\partial t}&=&\Psi_{t}^{\ast} \left( \mathcal{L}_{\theta^{\sharp}}g +  \frac{\partial g}{\partial t} \right)  \notag\\
	\Rightarrow \frac{\partial g^{\prime}}{\partial t}&=& -2Ric(g^{\prime})  \label{RicciMio}.
	\end{eqnarray}

The form $\theta^{\prime}$ is the mean curvature form of the pullback Haar system $\rho^{\prime}= \Psi^{\ast} \rho$. If we denote by $\beta^{\prime}$ the variation corresponding to $\rho^{\prime}$, we have
\begin{eqnarray*}
\frac{\partial \theta^{\prime}}{\partial t}= d\beta^{\prime}.
\end{eqnarray*}
On the other hand, we have
\begin{eqnarray*}
\frac{\partial \theta^{\prime}}{\partial t} &=& \Psi_{t}^{\ast} \left( \mathcal{L}_{\theta^{\sharp}}\theta +  \frac{\partial \theta}{\partial t} \right) \\
                                            &=& \Psi_{t}^{\ast} \left( d(|\theta|^2)  + d\beta \right) \\
																						&=& \Psi_{t}^{\ast} \left(  d(|\theta|^2 -R -\mathrm{div}(\theta)) \right) \\
																						&=& d \left(|\theta^{\prime}|^2 -R^{\prime} -\mathrm{div}(\theta^\prime) \right).
\end{eqnarray*}

Hence, we must have $\beta^{\prime} = |\theta^{\prime}|^2 -R^{\prime} -\mathrm{div}(\theta^\prime) + c(t)$ where $c(t)$ is some time dependent constant. Now, since the condition (\ref{impose}) implies that 
\begin{eqnarray*}
\int_{M} \,\phi^2d\mu_{g} = constant
\end{eqnarray*}
for any nonnegative cutoff function $\phi$ for the Haar system $\rho$, we will also have 
\begin{eqnarray*}
\int_{M} \, (\phi^\prime)^2d\mu_{g^\prime} = constant
\end{eqnarray*}
for any nonnegative cutoff function $\phi^{\prime}$ for the Haar system $\rho^{\prime}$. Differentiating the last equation with respect to time, we get
 \begin{eqnarray*}
\int_{M} \, \left( -R^{\prime}- \beta^{\prime} \right) (\phi^\prime)^2d\mu_{g^\prime} &=& 0 \\
\int_{M} \, \left( -R^{\prime}-|\theta^{\prime}|^2 + R^{\prime} +\mathrm{div}(\theta^\prime) -c(t) \right) (\phi^\prime)^2d\mu_{g^\prime} &=& 0 \\
\int_{M} \, \left( \underline{\mathrm{div}}(\theta^\prime) -c(t) \right) (\phi^\prime)^2d\mu_{g^\prime} &=& 0 \\
-c(t) \int_{M} \,(\phi^\prime)^2d\mu_{g^\prime} &=& 0
\end{eqnarray*}
So $c(t)$ is equal to zero and we get
\begin{eqnarray}
\beta^{\prime} = |\theta^{\prime}|^2 -R^{\prime} -\mathrm{div}(\theta^\prime). \label{betaFlow}
\end{eqnarray}
We can relate this to a backward heat equation as follows. We can fix a smooth family of representative Haar systems $\rho_{0}(t)$ which induce time dependent bijections $C^{\infty}_{\mathcal{G}}(M) \rightarrow \Xi$. We can then write the Haar system $\rho^{\prime}$ as $d\rho^{\prime}= e^{f^{\prime}} d\rho_{0}$ for some smooth time dependent $\mathcal{G}$-invariant  function $f^{\prime}$. If we denote by $\beta_{0}$ the variation of the fixed time dependent Haar system $\rho_{0}(t)$ and by $\theta_{0}$ the corresponding mean curvature form, then solving the equation~(\ref{betaFlow}) corresponds to solving the equation
\begin{eqnarray}
\frac{\partial f^{\prime}}{\partial t}&=&  |\nabla f^{\prime} + \theta_{0}|^2 -R^{\prime} -\mathrm{div} \left(\nabla f^{\prime} + \theta_{0} \right) -\beta_{0} \notag \\
                                      &=&  -\Delta f^{\prime} +|\nabla f^{\prime}|^2 + 2\langle \nabla f^{\prime}, \theta^{\prime} \rangle -\mathrm{div}(\theta_{0}) + |\theta_{0}|^2 -R^{\prime}  -\beta_{0}
\label{fevol}
\end{eqnarray}
which is a backward heat equation. Indeed, we can rewrite it as 
\begin{eqnarray}
\frac{\partial e^{-f^{\prime}}}{\partial t}&=& -\underline{\Delta}e^{-f^{\prime}} + \langle \nabla e^{-f^{\prime}}, \theta_{0} \rangle +(R^{\prime}+\underline{\mathrm{div}}(\theta_0) +\beta_0)e^{-f^{\prime}}. \label{fMio}
\end{eqnarray}

This can be solved as outlined in \cite{kleiner}. We first solve (\ref{RicciMio}) on some time interval $[t_1,t_2]$ and then solve (\ref{fMio}) backwards in time on $[t_1,t_2]$. The latter operation can be done by using the Uhlenbeck trick of Subsection~\ref{subsec:evolution-formulas}. We construct the space $\mathbf{F}$ with quotient space $Z=\mathbf{F}/\hat{\mathcal{G}}$. For each $t \in [t_1,t_2]$, we will have a commutative diagram 
\begin{eqnarray*}
 (\mathbf{F},\hat{g}(t)) &\xrightarrow{\hat{\sigma}}& (Z,\overline{g}(t)) \\
 \pi \downarrow &                & \downarrow \hat{\pi}   \\
 (M,g^{\prime}(t))   &\xrightarrow{\sigma}      & W. 
 \end{eqnarray*}
 The maps $\pi$ and $\hat{\pi}$ correspond to taking $O(n)$- quotients. Furthermore, the map $\pi$ is a $\mathcal{G}$-equivariant  Riemannian submersion and the map $\hat{\sigma}$ is an $O(n)$-equivariant Riemannian submersion for each $t \in [t_1,t_2]$. For simplicity, we shall assume that the fixed time dependent Haar system $\rho_{0}$ is the Haar system generated by the metric $g(t)$. We can view the functions $f$ and $\Lambda =(R^{\prime}+\underline{\mathrm{div}}(\theta_0) +\beta_0)$ as $O(n)$-invariant functions on $Z$. Recall that the form $\theta^{\prime}$ was induced by the mean curvature form of the Riemannian submersion $\hat{\sigma}$ which is a $\mathcal{G}$-basic and $O(n)$-basic one form. In particular, this form induces a corresponding $O(n)$-invariant form on the space $Z$. For simplicity, we will also denote this form by $\theta_{0}$. 

   It follows from \cite{gilkey}[Lemma 4.2.4] that, if $ q: (X_1,g_1) \rightarrow (X_0,g_0)$ is a Riemannian submersion and if $k$ is a smooth function on $X_0$, then
\begin{eqnarray*}
\Delta_{g_1}(q^{\ast} k) = q^{\ast}( \Delta_{g_0}k ) + \langle \alpha , \nabla(q^\ast k) \rangle
\end{eqnarray*}
where $\alpha$ is the mean curvature form of the submersion. Applying this to the submersions $\pi$ and $\hat{\sigma}$ and using the fact that the mean curvature form of the Riemannian submersion $\pi$ is equal to zero for each $t$, we see that Equation~(\ref{fMio}) now corresponds to the $O(n)$-invariant heat equation on $Z$
\begin{eqnarray}
\frac{\partial v}{\partial t}&=& -\Delta_{\overline{g}} v + \langle \nabla v, \theta_{0} \rangle + \Lambda v. \label{fMio2}
\end{eqnarray}
where $v= e^{-f}$. Since $Z$ is compact, we can start with a value $v_2=e^{-f_2}$ at time $t_2$ and solve this equation backward in time on $[t_1,t_2]$ to get an $O(n)$-invariant solution $u(t)$ for (\ref{fMio2}). We can show that $u(t) > 0$. So $f(t)= -\ln(v(t))$ is well defined and will be a solution for (\ref{fMio}). Using this method, we get the Haar system $\rho^{\prime}(t)$ having the variation specified by (\ref{betaFlow}). Finally, it is clear that we will have 
\begin{eqnarray*}
\frac{d}{dt}\mathcal{F}(g^{\prime}(t),\rho^{\prime}(t))=2 \int_{M} |Ric(g^{\prime}) + \frac{1}{2}\mathcal{L}_{(\theta^\prime)^{\sharp}}g^{\prime}|^2 (\phi^\prime)^2 d\mu_{g^\prime} \geq 0
\end{eqnarray*}
for any time dependent cutoff function $\phi^{\prime}$ compatible with $\rho^{\prime}$. This proves Theorem~\ref{monotonos}.

Just as in the closed manifold case, we have the following result.

\begin{ppr}\label{misere}
Given a $\mathcal{G}$-invariant metric $g$, there is a unique minimizer Haar system $\tilde{\rho}$ of $\mathcal{F}(g,\rho)$ under the constraint $\int_{M} \, \phi^2d\mu_{g}=1$ for any compatible cutoff function $\phi$.
\end{ppr}
\begin{proof}
We can rewrite the $\mathcal{F}$-functional as 
\begin{eqnarray*}
\mathcal{F}(g, \rho) =\int_{M} \left(R + 2 \mathrm{div}(\theta) -|\theta |^2 \right)\,  \phi^2d\mu_{g}  
\end{eqnarray*}
Just like before, we fix a representative Haar system $\rho_0$ and write the functional as:
\begin{eqnarray*}
\mathcal{F}(g, f) =\int_{M} \left(R + 2 \mathrm{div}(\theta_{0} + \nabla f) -|\theta_0 + \nabla f |^2 \right)\,  e^{-f} \phi_{0}^2d\mu_{g}  
\end{eqnarray*}
The constraint $\int_{M} \, \phi^2d\mu_{g}=1$ now becomes the condition $\int_{M} \, e^{-f} \phi_{0}^2d\mu_{g}=1$ for the $\mathcal{G}$-invariant smooth function $f$. For convenience again, we will take $\rho_0$ to be the Haar system generated by $g$ and as before we consider the orthonormal frame bundle $F_{g}$ generated by $g$ and the corresponding space of orbits $Z=F_{g}/\hat{\mathcal{G}}$  with Riemannian metric $\overline{g}$. Just as in the proof of \cite{kleiner}[Proposition 7.1], we set $\kappa=e^{-\frac{f}{2}}$.
We can write the $\mathcal{F}$-functional as an integral over $Z$.
\begin{eqnarray*}
\mathcal{F}= \int_{Z} \left( -4\Delta_{\overline{g}}(\kappa)  + B \kappa \right) d\mu_{\overline{g}}
\end{eqnarray*}
where $\kappa$ and $B= (2 \mathrm{div}(\theta_0) - |\theta_0|^2  + R)$ are viewed as $O(n)$-invariant functions on $Z$.

The constraint equation becomes $\int_{Z} \kappa^2 \,d\mu_{\overline{g}}=1$. The minimum of $\mathcal{F}$ is given by the smallest eigenvalue $\lambda$ of $-4 \Delta_{\overline{g}} + B$ and the minimzer $\tilde{\kappa} = e^{-\frac{\tilde{f}}{2}}$ is a corresponding normalized eigenvector. Since this operator is a Schr\"{o}dinger operator, there is a unique normalized positive eigenvector. Finally, the equation 
\begin{eqnarray*}
-4\Delta_{\overline{g}}(\kappa)  + B \kappa  = \lambda \kappa
\end{eqnarray*}
translates to
\begin{eqnarray*}
R + 2 \mathrm{div}(\theta) - |\theta|^2 = \lambda.
\end{eqnarray*}

\end{proof}
We define the $\lambda$-functional on $\mathcal{M}_{\mathcal{G}}$ by $\lambda(g)=\mathcal{F}(g,\tilde{\rho})$ where $\tilde{\rho}$ is the Haar system in the previous proposition.
\begin{ppr}
If $g(t)$ is a solution to the Ricci flow then $\lambda(g(t))$ is nondecreasing in $t$.
\end{ppr}
\begin{proof}
The proof is exactly the same as in the closed manifold case. 
Given a solution of the Ricci flow on a time interval $[t_1,t_2]$, we take $\rho(t_2)$ to be the minimizer $\tilde{\rho}(t_2)$. So, $\lambda(t_2)=\mathcal{F}(g(t_2),\tilde{\rho}(t_2))$. As we described right before Proposition~\ref{misere}, we can solve for the Haar system $\rho(t)$ backwards in time. It follows from Theorem~\ref{monotonos} that $\mathcal{F}(g(t_1),\rho(t_1)) \leq \mathcal{F}(g(t_2),\rho(t_2))$. By the definition of $\lambda$, we have $\lambda(t_1) \leq \mathcal{F}(g(t_1),f(t_1))$. Hence $\lambda(t_1) \leq \lambda(t_2)$. This completes the proof.
\end{proof}
We will define a steady breather as a Ricci flow solution on an interval $[t_1,t_2]$ such that $g(t_2)=\Psi^{\ast}g(t_1)$ for some self-equivalence $\Psi$ of $(\mathcal{G},M)$.
Since the $\lambda$-functional is invariant under self-equivalences, we can apply the previous proposition to get the following corollary.
\begin{crly}
A steady breather is a groupoid steady soliton.
\end{crly}

\begin{ppr}
A groupoid steady soliton on a closed groupoid $(\mathcal{G},M)$ with compact connected space of orbits $W$ is Ricci flat and the mean curvature class $\Theta$ is equal to zero.
\end{ppr}
\begin{proof}
The Haar system  $\rho(t)$ is the minimizer of $\mathcal{F}(g(t),.)$ for all $t$. Based on the proof of Proposition~\ref{misere}, we have
\begin{eqnarray*}
R + 2 \mathrm{div}(\theta) - |\theta|^2 = \lambda.
\end{eqnarray*}
Since this is a groupoid steady soliton, we know that $ Ric + \nabla \theta= 0$ and this gives us $R + \mathrm{div}(\theta) =0$ by taking the trace. So the above equation reduces to 
\begin{eqnarray*}
 \mathrm{div}(\theta) - |\theta|^2 = \lambda \\
\Rightarrow \underline{\mathrm{div}}(\theta) = \lambda
\end{eqnarray*}
Integrating this over $M$ with respect to the measure $\phi^2 d\mu_{g}$, we get that $\lambda =0$. 

  The groupoid soliton will give an eternal solution to the Ricci flow. By \cite{hilaire}[Lemma 6], an eternal Ricci flow solution on an \'etale groupoid with compact connected space of orbits must be Ricci flat. So $R=0$. Using the previous equations, we deduce that $\theta =0$. Hence, the clas $\Theta$ is equal to zero. This completes the proof.
 \end{proof}

\begin{rema}
Suppose that we have a Ricci soliton of the form 
\begin{eqnarray}
Ric(g) + \nabla(\alpha) + \frac{\epsilon}{2} g =0 \label{solitonG}
\end{eqnarray}
on a smooth manifold $M$ where $\alpha$ is a closed one form and $\epsilon$ is some constant. The cases $\epsilon > 0$, $\epsilon <0$ and $\epsilon =0$ will correspond to expanding, shrinking and steady solitons respectively. Just as in the case of gradient Ricci solitons, we take the divergence of the above equation and get
\begin{eqnarray*}
g^{ki}\nabla_{k}R_{ij} + g^{ki}\nabla_{k} \nabla_{i} \alpha_{j} =0 \\
\Rightarrow \frac{\nabla_{j}R}{2} + g^{ki} \nabla_{k}\nabla_{i} \alpha_{j} =0.
\end{eqnarray*}
Now, if we use the fact that $\nabla \alpha $ is symmetric, we get
\begin{eqnarray*}
\nabla_{k}\nabla_{i}\alpha _{j} &=&\nabla_{k}\nabla_{j}\alpha _{i} \\
                                &=& \nabla_{j}\nabla_{k}\alpha _{i} -R_{kjim}\alpha^{m}.
\end{eqnarray*}
Hence 
\begin{eqnarray*}
	 g^{ki} \nabla_{k}\nabla_{i} \alpha_{j} &=& \nabla_{j}(\mathrm{div}(\alpha)) + R_{jm}\alpha^{m} \\
	                                        &=& \nabla_{j}(\mathrm{div}(\alpha)) - \nabla_{j} \alpha_{m} \alpha^{m} -\frac{\epsilon}{2} \alpha_{j} \\
																					&=&  \nabla_{j}(\mathrm{div}(\alpha)) - \nabla_{j}(|\alpha|^2) -\frac{\epsilon}{2} \alpha_{j}.
\end{eqnarray*}	
 
So we now have
\begin{eqnarray}
\frac{\nabla_{j}R}{2} + \nabla_{j}(\mathrm{div}(\alpha)) - \nabla_{j}(|\alpha|^2) -\frac{\epsilon}{2} \alpha_{j} =0.
\end{eqnarray}
Taking the trace of (\ref{solitonG}) gives us $R + \mathrm{div}(\alpha) -\frac{n\epsilon}{2} =0$. We apply this to the above equation and end up with 
\begin{eqnarray}
\nabla\left( R + |\alpha|^2 \right) + \epsilon \alpha = 0  \label{solitonG2}
\end{eqnarray}
In the steady case ($\epsilon =0$), this reduces to 
\begin{eqnarray*}
\nabla\left( R + |\alpha|^2 \right) = 0 \\
\Rightarrow R + |\alpha|^2  = constant.
\end{eqnarray*}

Notice that, if we take $\alpha$ to be the mean curvature form $\theta$, then the left-hand side of the above equation is the integrand of the $\mathcal{F}$-functional that we defined in the beginning of this section.  

In the case where $\epsilon \neq 0$, Equation~(\ref{solitonG2}) implies that the closed form $\alpha$ is actually exact and this would reduce to the familiar gradient Ricci soliton setting. So we can think of the class $\Theta$ as a an obstruction to defining a $\mathcal{W}$-functional on Riemannian groupoids.
\end{rema}
We end this section by showing that, as expected, there is a differential Harnack inequality associated to the system of equations (\ref{monotonoseq}). More precisely, we fix as before a time dependent Haar system $\rho_{0}(t)$ with variation determined by a $\mathcal{G}$-invariant function $\beta_{0}$ and with mean curvature form $\theta_{0}$. We write the solution $\rho(t)$ of (\ref{monotonoseq}) as $d\rho= e^{f} d\rho_{0}$ so that the evolution equation of $\rho$ is now equivalent to 
\begin{eqnarray*}
\frac{\partial f}{\partial t}= |\theta_{0} + \nabla f|^2 - R - \mathrm{div}(\theta_{0} + \nabla f) - \beta_{0} \\
\end{eqnarray*}
Setting $v=e^{-f}$, we can rewrite the above equation as
\begin{eqnarray}
 P(v) = 0  \label{Poperator}
\end{eqnarray}
 where $P$ is the differential operator
\begin{eqnarray}
P= \frac{\partial}{\partial t} -2 \langle \theta_{0}, \nabla(.) \rangle  - \left( \mathrm{div}(\theta_{0}) - |\theta_{0}|^2 + \beta_{0} \right).
\end{eqnarray}
We have the following result.
\begin{ppr}
Supose the $\mathcal{G}$-invariant function $v=e^{-f}$ satisfies (\ref{Poperator}) where the $\mathcal{G}$-invariant metric varies under the Ricci flow. Then
\begin{eqnarray*}
P \left( (R+ 2 \mathrm{div}(\theta_{0} + \nabla f) -|\theta_{0} + \nabla f|^2) v \right) = 2|Ric + \nabla \nabla f + \nabla \theta_{0}|^2 v \geq 0.
\end{eqnarray*}
\end{ppr}
\begin{proof}
We can show that the term
\begin{eqnarray*}
 v^{-1}P \left( (R+ 2 \mathrm{div}(\theta_{0} + \nabla f) -|\theta_{0} + \nabla f|^2) v \right)
\end{eqnarray*}
 is equal to 
\begin{eqnarray}
-2 \langle \nabla H, \theta_{0} + \nabla f \rangle + \left( \frac{\partial }{\partial t} + \Delta \right) H \label{Ppart1} 
  \end{eqnarray}                                                                               
where 
\begin{eqnarray*}
H = R+ 2 \mathrm{div}(\theta_{0} + \nabla f)-|\theta_{0} + \nabla f|^2.
\end{eqnarray*}
We know that
\begin{eqnarray}
\left( \frac{\partial }{\partial t} + \Delta \right)R= \Delta R + 2 |Ric|^2. \label{mize1}
\end{eqnarray}
Given a variation $\delta g =v$, the variation of the divergence of a one form $\alpha$ is given by 
\begin{eqnarray*}
\delta (\mathrm{div}(\alpha))= -\langle v, \nabla \alpha \rangle + \langle  \left(\mathrm{div}(v) - \frac{\nabla{V}}{2} \right),  \alpha \rangle.
\end{eqnarray*}
Applying this, we get 
\begin{eqnarray*}
\left( \frac{\partial }{\partial t} + \Delta \right)(2 \mathrm{div}(\theta_{0} + \nabla f))&=&  4 \langle Ric, \nabla(\theta_{0} + \nabla f) \rangle \\
                   & & + 2 \Delta( \mathrm{div}(\theta_{0} + \nabla f)) + 2\Delta \left(\frac{\partial f }{\partial t} + \beta_{0}\right)
\end{eqnarray*}
which reduces to 
\begin{eqnarray}
\left( \frac{\partial }{\partial t} + \Delta \right)(2 \mathrm{div}(\theta_{0} + \nabla f))&=& 4 \langle Ric, \nabla(\theta_{0} + \nabla f) \rangle -2\Delta R \label{mize2}\\
                                                                                           & & + 2 \Delta |\theta_{0} + \nabla f|^2. \notag
\end{eqnarray}
The last term of $H$ gives us 
\begin{eqnarray*}
\left( \frac{\partial }{\partial t} + \Delta \right)(-|\theta_{0} + \nabla f|^2) &=& -2Ric(\theta_{0} + \nabla f,\theta_{0} +  \nabla f) \\
              & &-2 \langle \theta_{0} + \nabla f, \nabla\left(\frac{\partial f }{\partial t} + \beta_{0} \right) \rangle  - \Delta|\theta_{0} + \nabla f|^2
\end{eqnarray*}
which becomes
\begin{eqnarray*}
\left( \frac{\partial }{\partial t} + \Delta \right)(-|\theta_{0} + \nabla f|^2) &=& -2Ric(\theta_{0} + \nabla f,\theta_{0} + \nabla f) - \Delta|\theta_{0} + \nabla f|^2 \\
                                                                                  & & -2 \langle \theta_{0} + \nabla f, \nabla\left( |\theta_{0} + \nabla f|^2 -R \right) \rangle \\
		                                                                               & &+ 2 \langle \theta_{0} + \nabla f,\mathrm{div}(\theta_{0} + \nabla f) \rangle.
\end{eqnarray*}
Combining this with (\ref{mize1}) and (\ref{mize2}), we get 
\begin{eqnarray*}
\left( \frac{\partial }{\partial t} + \Delta \right)H &=& 2|Ric|^2 + 4 \langle Ric, \nabla(\theta_{0} + \nabla f) \rangle -2Ric(\theta_{0} + \nabla f,\theta_{0} + \nabla f)\\
                & & \Delta|\theta_{0} + \nabla f|^2 -2\langle \theta_{0} + \nabla f, \nabla\left( |\theta_{0} + \nabla f|^2 -R \right) \rangle \\
								& & 2 \langle \theta_{0} + \nabla f,\mathrm{div}(\theta_{0} + \nabla f) \rangle.
\end{eqnarray*}
For any closed  one form $\alpha$, we have
\begin{eqnarray*}
\Delta |\alpha|^2  = 2|\nabla \alpha |^2 + 2 Ric(\alpha, \alpha) + 2 \langle \nabla( \mathrm{div}(\alpha)), \alpha \rangle.
\end{eqnarray*}
Therefore
\begin{eqnarray*}
\left( \frac{\partial }{\partial t} + \Delta \right)H = 2|Ric + \nabla \nabla f + \nabla \theta_{0}|^2 +2 \langle \nabla H, \theta_{0} + \nabla f \rangle
\end{eqnarray*}
which gives us
\begin{eqnarray*}
v^{-1}P(Hv) &=& -2 \langle \nabla H, \theta_{0} + \nabla f \rangle + \left( \frac{\partial }{\partial t} + \Delta \right) H \\
            &=&  2|Ric + \nabla \nabla f + \nabla \theta_{0}|^2.
\end{eqnarray*}
This completes the proof.
\end{proof}

\appendix
\section{The 2-dimensional case}
In this appendix, we look at closed Riemannian groupoids $(\mathcal{G},M,g)$ where $\mathrm{dim} M=2$. If we consider the local models determined by the quintuples $(\mathfrak{g},K,i,Ad,T)$ described in Subsection~\ref{localModel}, we see that, if the structural Lie algebra has dimension at least two, then we are dealing with a locally homogeneous surface. The groupoid is then developable and the $\mathcal{G}$-invariant metric will have constant sectional curvature. The Ricci flow is pretty simple in this case.  On the other hand, if the structural Lie algebra is trivial, then the groupoid is the groupoid of germs of change of charts of a $2$-dimensional orbifold. Except for a few cases (the bad $2$-orbifolds), they will be developable and the Ricci flow is well understood in this setting. Finally, if $\mathfrak{g}$ is isomorphic to $\mathbb{R}$, then we should consider two cases:
 \begin{enumerate}
\item The space of orbits $W$ is a closed interval.
\item The space of orbits $W$ is a circle.
\end{enumerate}
The first case will be weakly equivalent to an $O(2)$ or $SO(2)$ action on the $2$-sphere.  
  By \cite{lott2}[Section 5], the second case is equivalent to the groupoid generated by the action of $\mathbb{R} \rtimes \mathbb{Z}$ on $\mathbb{R} \times \mathbb{R}$ given       by 
\begin{eqnarray*}
(q,n) \cdot (x,y) = (c^{n}x + q, y + n)
\end{eqnarray*}
where $c$ is some fixed nonzero constant. The coordinate $y$ induces a coordinate on the space of orbits which is indeed a circle. It is also clear that the case $c=1$ is equivalent to looking at a torus. In general, after possibly changing the coordinate $y$, we can assume that the invariant Riemannian metric $g$ has the form
\begin{eqnarray}
g(x,y)=e^{2k(y)} dx^2 + e^{2u(y)}dy^2 \label{circlecase}
\end{eqnarray}
where $u(y+1)=u(y)$ and $k(y+1)=k(y) + \lambda $ for some real number $\lambda$ ($\lambda=0$ corresponds to the torus case). The term $e^{2u(y)}dy^2$ inducces a metric $\overline{g}$ on the quotient space. The function $k$ cannot be viewed as a function on the quotient space, but the differential $dk$ is well-defined on this space.
 By doing a simple computation, we can show that the scalar curvature $R$ of the metric $g$ is given by 
\begin{eqnarray}
R &=& -2 \left( \Delta_{\overline{g}}k + |\nabla k|^2 \right) \\
\Rightarrow R &=& -2 \left( \frac{d^2k}{ds^2} + \left(\frac{dk}{ds} \right)^2 \right) \notag
\end{eqnarray}
where $s$ denotes the arc length coordinate on the quotient circle. In terms of the coordinate $s$, we will have $k(s+L) =k(s) + \lambda$ where $L$ is the length of the circle. Finding the metrics $g$ of constant scalar curvature reduces to solving the differential equation 
\begin{eqnarray*}
 \left( \frac{d^2k}{ds^2} + \left(\frac{dk}{ds} \right)^2 \right) = -\frac{c}{2}.
\end{eqnarray*}
We cannot have positive constant scalar curvature. For if $c >0$, the above equation would imply that $\frac{d^2k}{ds^2} <0$. So $\frac{dk}{ds}$ would be strictly decreasing thus contradicting the fact that it is a periodic function of $s$. This periodicity condition also implies that, in the case $c=0$, we must have $k(s)= constant$ (so $\lambda =0)$ and, in the case $c <0$, the metric $g$ is of the form 
\begin{eqnarray*}
ds^2 + e^{2 \lambda s} dx^2.
\end{eqnarray*}
So the constant negative scalar curvature case corresponds to the hyperbolic cusp. 

The Ricci flow equation applied to Riemannian metrics of the form (\ref{circlecase}) reduces to
\begin{eqnarray*}
\frac{\partial u}{\partial t}=\frac{\partial k}{\partial t}= \Delta_{\overline{g}}k + |\nabla k|^2.
\end{eqnarray*}
Note that the evolution equation for $k$ can be written in the simpler form
\begin{eqnarray*}
\frac{\partial e^{k}}{\partial t} = \Delta_{\overline{g}} e^{k}.
\end{eqnarray*}
But this will not be very useful since, as we mentioned earlier, $e^{k}$ cannot be viewed as function on the quotient circle.

By doing a simple computation, we can show that the evolution equation of $|\nabla k|^2$ is given by
\begin{eqnarray}
\left( \frac{\partial}{\partial t} -\Delta_{\overline{g}} \right) |\nabla k|^2 = -2\left((|\nabla k|^2 +(\Delta_{\overline{g}}k)^2\right) + \langle \nabla k, \nabla (|\nabla k|^2) \rangle \label{Kcontrol}
\end{eqnarray}
If $C$ is the maximum of the function $|\nabla k|^2$ at time $t=0$, then, by the maximum principle, $|\nabla k|^2(,t) \leq C$ as long as the solution exists.
We can also show that the function $H= (\Delta_{\overline{g}}k)^2$ satisfies
\begin{eqnarray}
\left( \frac{\partial}{\partial t} -\Delta_{\overline{g}} \right) H \leq -8 |\nabla k|^2 H + \langle \nabla k, \nabla H \rangle. \label{Hcontrol}
\end{eqnarray}
This means that $\Delta_{\overline{g}}k$ stays bounded. Since $R=-2 \Delta_{g}k= -2 \left( \underline{\Delta}k + |\nabla k|^2 \right)$, the scalar curvature is bounded. So the Ricci flow exists for all time. If we assume in addition that $|\nabla k| \neq 0$ at time $t=0$, then we can extract more information from the evolution equation (\ref{Kcontrol}). We can show that the evolution equation of $|\nabla k|$ is given by
\begin{eqnarray}
\left( \frac{\partial}{\partial t} -\Delta_{g} \right) |\nabla k| = -|\nabla k|^3 + \langle \nabla k, \nabla (|\nabla k|) \rangle.  \label{key2}
\end{eqnarray}
So, if $ 0 < \alpha < |\nabla k|^2 < \beta  $ at time $t=0$, we can apply the maximum/ minimum principles to deduce that 
\begin{eqnarray}
\frac{\alpha}{2\alpha t +1 } \leq  |\nabla k|^2 \leq \frac{\beta}{2\beta t +1}. \label{inequ1}
\end{eqnarray}
We can then use the derived lower bound on $|\nabla k|^2$ and consider the equation (\ref{Hcontrol}) to derive the inequality
\begin{eqnarray}
H \leq \frac{C}{(2\alpha t +1)^4} \label{inequ2}
\end{eqnarray}
for some positive constant $C$. We now consider the new family of metrics $ \hat{g}(t)=\frac{g(t)}{t}$. This will be of the form  $\hat{g}(t) = e^{2\hat{k}(y,t)} dx^2 + e^{2\hat{u}(y,t)}dy^2$ where
$|\nabla \hat{k} |^2 = t |\nabla k|^2 $ and the corresponding function $\hat{H}$, as defined above, is given by $\hat{H}=t^2 H$. It follows from (\ref{inequ1}) and (\ref{inequ2}) that 
$|\nabla \hat{k} |^2 \rightarrow \frac{1}{2}$ and $ \hat{H} \rightarrow 0 $ as $t \rightarrow \infty $. Skipping some details, this is saying that $\hat{g}(t)$ converges as $t \rightarrow \infty$ to a metric $g _{\infty} =e^{2 k_{\infty}(y,t)} dx^2 + e^{2 u_{\infty}(y,t)}dy^2$ of constant scalar curvature
\begin{eqnarray*}
R_{g_{\infty}}= -2 \Delta_{g_{\infty}} k_{\infty} = -2 |\nabla k_{\infty} |^2= -1.
\end{eqnarray*}
So it's the hyperbolic cusp.

  \end{document}